\documentclass[oneside,reqno,10pt]{compositio}

\usepackage{amsmath}
\usepackage{amssymb}
\usepackage{amscd}
\usepackage{color}

\newtheorem{theorem}{Theorem}[section]
\newtheorem{prop}[theorem]{Proposition}

\newtheorem{remark}[theorem]{Remark}
\newtheorem{lemma}[theorem]{Lemma}

\newtheorem{conjecture}[theorem]{Conjecture}

\def\a{{\alpha}}
\def\b{{\beta}}
\def\G{\Gamma}
\def\l{\lambda}
\def\v{\varpi}
\def\ep{\varepsilon}
\def\vp{\varphi}
\def\th{\theta}
\def\Th{\Theta}

\def\R{{\mathbb{R}}}
\def\Q{{\mathbb{Q}}}
\def\C{{\mathbb{C}}}
\def\Z{{\mathbb{Z}}}

\def\D{{\rm D}}
\def\A{{\mathbb{A}}}
\def\O{{\mathcal{O}}}
\def\K{{\rm K}}
\def\B{{\rm B}}
\def\cA{{\mathcal{A}}}
\def\Sc{{\mathcal{S}}}
\def\W{{\mathcal{W}}}

\def\o{\mathfrak{o}}
\def\p{{\mathfrak{p}}}

\def\i{\infty}
\def\ol{\overline}
\def\t{\times}
\def\ot{\otimes}
\def\bt{\boxtimes}
\def\bs{\backslash}
\def\GL{{{\rm GL}}}
\def\PGL{{{\rm PGL}}}
\def\SL{{{\rm SL}}}
\def\GSp{{{\rm GSp}}}
\def\Sp{{{\rm Sp}}}
\def\GSO{{{\rm GSO}}}
\def\GO{{{\rm GO}}}
\def\O{{{\rm O}}}

\theoremstyle{definition}
\numberwithin{equation}{section}
\newtheorem{thm}{Theorem}[section]

\newtheorem{Rem}[thm]{Remark}

\newcommand{\bQ}{\overline{\mathbb{Q}}}

\newcommand{\ds}{\displaystyle}

\newcommand{\lra}{\longrightarrow}

\begin{document}
\title{Endoscopic lifts to the Siegel 
modular threefold related to Klein's cubic threefold}
\author{Takeo Okazaki}
\email{okazaki@cc.nara-wu.ac.jp}
\address{Department of Mathematics, Faculty of Science, Nara Woman University,
Kitauoyahigashi-machi, Nara 630-8506, Japan.}
\author{Takuya Yamauchi}
\email{yamauchi@las.osakafu-u.ac.jp}
\address{Faculty of Liberal Arts and Sciences, Osaka Prefecture University
1-1 Gakuen-cho, Nakaku, Sakai, Osaka 599-8531, Japan.}
\dedication{Dedicated to Professor Takayuki Oda 
on his 60-th birthday.}
\thanks{
The first author would like to express his appreciation to Professor Ralf Schmidt for his warmhearted explanation on his results and to Professor Tomonori Moriyama for his kind various advice.
The second author would like to express his appreciation to 
Professor Klaus Hulek and Takayuki Oda for valuable comments. 
They would like to thank the referees for many valuable comments, in particular suggesting us take Arthur's 
conjecture into acount. 
The first author is supported by
the Japan Society for the Promotion of Science Research 
Fellowships for Young Scientists and the second author is partially supported by JSPS Grant-in-Aid for Scientific Research No.19740017 
and JSPS Core-to-Core Program No.18005.}

\maketitle

{\footnotesize{\bf Abstract.-} 
Let $\mathcal{A}^{{\rm lev}}_{11}$ be 
the moduli space of $(1,11)$-polarized abelian surfaces 
with a canonical level structure. Let $\chi$ be a primitive character 
of order 5 with conductor 11. 
In this paper we construct five endoscopic lifts 
$\Pi_i,0\le i\le 4$ from two elliptic modular forms 
$f\otimes\chi^i$ of weight 2 and 
$g\otimes\chi^i$ of weight 4 with complex multiplication by $\Q(\sqrt{-11})$ 
such that 
${\Pi_i}_\infty$ gives a non-holomorphic 
differential form on $\mathcal{A}^{{\rm lev}}_{11}$ 
for each $i$, 
$0\le i\le 4$.  
Then their spinor $L$-functions 
are of form $L(s-1,f\otimes\chi^i)L(s,g\otimes\chi^i)$ such that  
$L(s,g\otimes\chi^i)$ does not appear in the $L$-function of $\mathcal{A}^{{\rm lev}}_{11}$ for any $i$, 
$0\le i\le 4$. 
The existence of such lifts is motivated by the computation of the $L$-function of Klein's cubic threefold which is a birational smooth model of $\mathcal{A}^{{\rm lev}}_{11}$.  
}

\section{Introduction}
Let $\mathcal{A}^{{\rm lev}}_{11}$ be 
the moduli space of $(1,11)$-polarized abelian surfaces 
with a canonical level structure (see \cite{hkw} 
for details). It is obtained as a quotient of the Siegel upper half plane $\mathbb{H}_2$ 
by the arithmetic subgroup $\K(11)^{{\rm lev}}$ 
in the symplectic group $\Sp_2(\Z) \subset \GL_4(\Z)$ 
(see Section 3 for the definition of  $\K(11)^{{\rm lev}}$). 
The congruence subgroup 
$$\Gamma(11):=\{\gamma\in \Sp_2(\Z)\ |\ \gamma\equiv 1_4\ {\rm mod}\ 11 \}\subset \Sp_2(\Z)$$ 
is a normal subgroup of $g\K(11)^{{\rm lev}}g^{-1}$ where 
$g={\rm diag}(1,1,11,11)$. 
Then we have the Galois covering 
$\pi:S_{\Gamma(11)}:=\Gamma(11)\backslash\mathbb{H}
\lra \mathcal{A}^{{\rm lev}}_{11}$ of moduli spaces with 
the Galois group $G:=g\K(11)^{{\rm lev}}g^{-1}/
\Gamma(11)$. Since $\Gamma(11)$ 
is torsion free, $S_{\Gamma(11)}$ is a quasi-projective smooth variety. 
On the other hand,  $\K(11)^{{\rm lev}}$ has many torsion points, so 
$\mathcal{A}^{{\rm lev}}_{11}$ is a quasi-projective variety but it 
has many quotient singularities. 
We can view these moduli spaces as varieties defined over $\Q$ if we 
consider the level structure \'etale locally (cf. \cite{mumford}). Note that it is easy to see that canonical models of these varieties are both 
defined over $\Q(\zeta_{11^2})$. This fact follows from the moduli interpretation and a basic knowledge of 
Weil pairing.    
In \cite{g&p}, 
Mark Gross and Popescu determined 
a birational smooth model $X$ of $\mathcal{A}^{{\rm lev}}_{11}$ 
by a significant method. 
More precisely they constructed a birational map 
$\mathcal{A}^{{\rm lev}}_{11}\longrightarrow X$ over $\Q$ where 
$X$ is defined by a simple equation in projective space 
$\mathbb{P}^4_\Q$: 
$$X:x_0^2x_1+x^2_1x_2+x^2_2x_3+x^2_3x_4+x^2_4x_0=0.$$  
The variety $X$ is called Klein's cubic threefold. 
Since $X$ is a smooth cubic threefold (hence a Fano threefold), 
it has Hodge numbers 
\begin{eqnarray}
h^{0,0}=1, 
h^{1,0}=h^{0,1}=h^{2,0}=h^{0,2}=h^{3,0}=h^{0,3}=0, h^{1,1}=1,
h^{2,1}=h^{1,2}=5. \label{eqn:hstK}
\end{eqnarray}
In particular, we see that the local $L$-factor of $X$ is 
of degree 10. 

In this paper we first compute the $L$-function of $X$. 
As a result we have:
\begin{theorem} Let $\ell$ be a prime number. 
Let $f$ be the elliptic modular form of weight 2 for 
$\Gamma_0(11^2)$ with 
complex multiplication by the ring of integers of 
$\Q(\sqrt{-11})$. Let $\chi$ be  
a primitive character of order 5 with conductor 11. 
Then we have $L(s,H^3_{{\rm et}}(X_{\bQ},\Q_\ell))=\displaystyle\prod_{i=0}^4
L(s-1,f \ot \chi^i)$. In particular, the left hand side is 
independent to any choice of $\ell$. 
\end{theorem}
By Theorem 1.1 and the Hodge type of $X$, 
it is quite natural to 
predict the existence of non-holomorphic differential forms on $\mathcal{A}^{{\rm lev}}_{11}$ 
of Hodge type $(2,1)$ of which corresponding automorphic forms 
on GSp$_2(\A)$ are liftings related to 
the elliptic modular form $f$ 
(we will discuss on the (1,1)-part in another paper). 
In a similar situation, in \cite{o&y}, 
authors treated a unique holomorphic differential 3-form on some Siegel threefold.
Since the space of holomorphic 3-forms  is a birational invariant, we can study this form by using an 
explicit birational model of the Siegel threefold in \cite{o&y}. 
In this paper we treat non-holomorphic differential forms.
However their space is not a birational invariant. 
Therefore we cannot directly construct  
non-holomorphic differential forms on $\mathcal{A}^{{\rm lev}}_{11}$ from 
them on $X$.   
Note that since $H^{3,0}(X)=0$,  
there does not exist any holomorphic 
differential 3-form on 
$\mathcal{A}^{{\rm lev}}_{11}$ which extends to one on any 
birational smooth model of $\mathcal{A}^{{\rm lev}}_{11}$.  
So this case will give a first and fascinating example 
to spur authors on an explicit construction of 
non-holomorphic differential forms of which corresponding 
automorphic forms on GSp$_2(\A)$ are liftings related to 
the elliptic modular form $f$. 

We now explain the second main result. 
For a unitary irreducible automorphic representation $\pi$ of GL$_2(\A)$, 
let $L(s,\pi)$ be the automorphic $L$-function of $\pi$ (see \cite{JL}) and if 
$\pi$ is attached to an elliptic modular form $h$ of weight $k$, then 
$L(s+(k-1)/2,h) =L(s,\pi)$ by definition.  
Let $\mu$ be the gr\"o{\ss}encharacter of $K = \Q(\sqrt{-11})$ associated to $f$ and 
$g$ be the elliptic modular form associated to  $\mu^3$.
Note that $g$ is of weight 4 and of level $11^2$. 
Then we have: 
\begin{theorem}$($Theorem 3.2$)$ 
There are irreducible, cuspidal, automorphic, globally generic representations $\Pi_{i}, 0\le i\le 4$ of $\GSp_2(\A)$ with $\Pi_{i,\i}$ being a discrete series whose Blattner parameter is $(3,-1)$.
Each $\Pi_i$ satisfies the following properties:

\medskip
\noindent
$($i$)$ 
There is a non-zero right $\K(11)_{\A}^{{\rm lev}}$-invariant automorphic form $F_i \in \Pi_i$. 

\medskip
\noindent
$($ii$)$ 
The spinor $L$-function of $\Pi_i$  
(which is denoted by $L(s,\Pi_i,{\rm spin})$) is $L(s,\pi_f \ot \chi^i)L(s,\pi_g \ot\chi^i)$, where $\pi_f$ 
$($resp. $\pi_g)$ is the unitary irreducible 
automorphic representation attached to $f$ $($resp. $g)$. 
$($see Novodvorsky \cite{Nov} for the definition of the spinor $L$-function of a generic representation$)$. Note that 
$L(s-3/2,\Pi_i,{\rm spin})=L(s-1,f\otimes\chi^i)L(s,g\otimes\chi^i)$. 
\end{theorem}

The strategy of the construction of $\Pi_i$ is as follows. 
In our case, by combining several facts, we guess that $\Pi_i$ is a weak endoscopic lift in the sense of \cite{W}.  
By results of Kudla, Rallis, and Soudry \cite{K-R-S}, and Roberts \cite{Ro}, we know that such $\Pi_i$ is given by a 
$\th$-lift $\Th(\pi_1 \bt \pi_2)$ of a pair $(\pi_1,\pi_2)$ of two irreducible cuspidal automorphic representations of $\GL_2(\A)$ 
(Here we identify $(\pi_1, \pi_2)$ with an automorphic representation of $\GSO_{2,2}(\A)$). 
It is natural to guess that $\pi_1$ should be $\pi_f \ot \chi^i$.
We also have to find a candidate of $\pi_2$. 
After trial and error  
(but there is no 
precise evidence) 
we decide $\pi_2= \pi_g \ot \chi^i$ and 
by choosing a suitable Schwartz-Bruhat function for the $\th$-lift, 
realize a non-zero right $\K(11)_{\A}^{\rm lev}$-invariant vector $F_i \in \Th((\pi_f \ot \chi^i) \bt (\pi_g \ot \chi^i))$ in Theorem 1.2.   

We should discuss a comparison of differential forms and $L$-functions 
between $X$ and $\mathcal{A}^{{\rm lev}}_{11}$. 
Let $H^3_{{\rm cusp}}(\mathcal{A}^{{\rm lev}}_{11},\C)$ be 
the cuspidal part of the Betti cohomology $H^3(\mathcal{A}^{{\rm lev}}_{11},\C)$ (see Section 4).
Then by combining above two theorems, we have:
\begin{theorem} The notations being as above. Let $E_{i}$ be the elliptic curve attached to $f\otimes\chi^i$ 
for each $i\in\{0,1,2,3,4\}$. We fix a non-zero holomorphic 1-form 
$\omega_i$ on $E_i$. 

\medskip
\noindent
$($i$)$ There exists a linear map $H^3_{{\rm dR}}(X_\C)\simeq \displaystyle\bigoplus_{0\le i\le 4}
H^3_{{\rm dR}}({E_{i}}_\C\times\mathbb{P}^2_\C)
\longrightarrow H^3_{{\rm cusp}}(\mathcal{A}^{{\rm lev}}_{11},\C)$ which 
is injective and preserves 
the Hodge structure. Here the second map is given by  
$\omega_i\otimes\Lambda\mapsto {F_i}_\infty,\ 
\overline{\omega}_i\otimes\Lambda\mapsto \overline{{F_i}_\infty}$ where $\Lambda$ is a generator of $H^{1,1}(\mathbb{P}^2_\C)$.

\medskip
\noindent
(ii) for any $i$, the $L$-function of $\Pi_i$, more precisely $\pi_g\otimes\chi^i$ does not contribute to the $L$-function of 
the parabolic cohomology $H^3_{{\rm et},!}({\mathcal{A}^{{\rm lev}}_{11}}_{\overline{\Q}},\Q_\ell)$ 
defined by 
the image of the natural map from 
the $\ell$-adic \'etale cohomology with compact support 
$H^3_{{\rm et},c}({\mathcal{A}^{{\rm lev}}_{11}}_{\overline{\Q}},\Q_\ell)$ 
to $H^3_{{\rm et}}({\mathcal{A}^{{\rm lev}}_{11}}_{\overline{\Q}},\Q_\ell)$. 
\end{theorem}
Recall that 
${\mathcal{A}^{{\rm lev}}_{11}}$ is a singular variety. 
So we do not know a priori whether 
$H^3_{{\rm et},!}({\mathcal{A}^{{\rm lev}}_{11}}_{\overline{\Q}},\Q_\ell)$ is 
pure of weight 3. 
We check this as follows. With the notations of the begining of the 
introduction, we have  by transfer theorem,  
$$H^3_{{\rm et},c}({{\mathcal{A}^{{\rm lev}}_{11}}}_{\bQ},\Q_\ell)
\stackrel{\sim}{\lra} 
H^3_{{\rm et},c}({S_{\Gamma(11)}}_{\bQ},\Q_\ell)^G
\hookrightarrow 
H^3_{{\rm et},c}({S_{\Gamma(11)}}_{\bQ},\Q_\ell).$$
It is easy to see that 
this map is compatible with the pull-pack $\pi^\ast$. Note that $\pi^\ast$ 
 is injective because $\pi$ is a finite map. 
Hence $H^3_{{\rm et},!}({\mathcal{A}^{{\rm lev}}_{11}}_{\overline{\Q}},\Q_\ell)$ 
is a Gal$(\bQ/\Q)$-submodule of 
$H^3_{{\rm et},!}({S_{\Gamma(11)}}_{\overline{\Q}},\Q_\ell)$. 
Since $S_{\Gamma(11)}$ is smooth, by Corollaire (3.3.6) of \cite{deligne}, 
the cohomology $H^3_{{\rm et},!}({S_{\Gamma(11)}}_{\overline{\Q}},\Q_\ell)$ is pure of weight 3 and so is $H^3_{{\rm et},!}({\mathcal{A}^{{\rm lev}}_{11}}_{\overline{\Q}},\Q_\ell)$. 
\begin{remark}
Theorem 1.3-(i) does not make sense if we do not tell about 
$\Q$-Hodge structures or period Lattices in those cohomologies. 
The period lattice $H^3(X,\Z)$ is determined by Roulleau \cite{rou}. 
If we determine  $H^3_!(\mathcal{A}^{{\rm lev}}_{11},\Z)$ and show that 
it is quasi-isomorphic to $H^3(X,\Z)$ under the map of Theorem 1.3-(i), 
we will be able to prove that $L(s,H^3_{{\rm et}}(X_{\overline{L}},\Q_\ell))$ occurs in 
$L(s,H^3_{{\rm et},!}({\mathcal{A}^{{\rm lev}}_{11}}_{\overline{L}},\Q_\ell))$ 
for some number field $L$.  
\end{remark}
Let $\G_{\A}$ be an open compact subgroup of $\GSp_2(\A)$ such that 
$\Gamma:=\G_{\A} \cap \Sp_2(\Q)$ is an arithmetic subgroup of $\Sp_2(\Q)$ 
(not necessary torsion free).  
Let us denote by $S_\Gamma = \G \bs \mathbb{H}_2$ the corresponding Siegel modular threefold. 
Take a Galois cover 
$\pi:S_{\G'}\lra S_\G$ with the Galois group $G$ 
so that $S_{\G'}$ is a fine moduli space. For an irreducible representation $L_{a,b}$ over $\C$ of $\Sp_2(\R)$ with dominant highest weight $(a,b)$ with $a \ge b \ge 0$ (we assume that $L_{a,b}$ 
comes from an irreducible algebraic representation of 
$\Sp_2$), we can define the local system $\mathcal{L}_{a,b}/\C$ and the ($\ell$-adic) \'etale local system $\mathcal{L}^{{\rm et}}_{a,b}$  on $S_{\G'}$ 
associated to $L_{a,b}$ by using 
the Hodge bundle on $S_{\G'}$ (cf. Section 1 of \cite{W}). Put $w=a+b+3$. 
Let ${\rm Gr}^W_wH^3_{{\rm betti}}({S_{\G'}}_\C,\mathcal{L}_{a,b})$ be 
the graded quotient of degree $w$ of a mixed Hodge structure on 
$H^3_{{\rm betti}}({S_{\G'}}_\C,\mathcal{L}_{a,b})$. We define 
${\rm Gr}^W_wH^3_{{\rm betti}}({S_{\G}}_\C,\mathcal{L}_{a,b}):=
{\rm Gr}^W_wH^3_{{\rm betti}}({S_{\G'}}_\C,\mathcal{L}_{a,b})^G$ and 
${\rm Gr}^W_wH^3_{{\rm et}}({S_{\G}}_{\bQ},\mathcal{L}^{{\rm et}}_{a,b}):=
{\rm Gr}^W_wH^3_{{\rm et}}({S_{\G'}}_{\bQ},\mathcal{L}^{{\rm et}}_{a,b})^G$. 

In what follows, we will discuss our results with  
known results or conjectures about irreducible cuspidal automorphic 
representation $\Pi$ of $\GSp_2(\A)$ which arises from a differential form of 
${\rm Gr}_w^W H^3_{{\rm betti}}({S_\Gamma}_\C,\mathcal{L}_{a,b})$. 
Assume that $\Pi = \ot_v \Pi_v$ is weakly equivalent to a 
representation of multiplicity one.
Let $L_{\Pi}$ be the set of irreducible cuspidal automorphic representations 
$\Pi' = \ot_v \Pi_v'$ of $\GSp_2(\A)$ such that every $\Pi_v'$ belongs to 
the local $L$-packet of $\Pi_v$ for each place $v$ of $\Q$.
See \cite{GT} and \cite{Paul} for the definition of the local $L$-packet.
By Theorem 2.1 of \cite{Onote}, if $\Pi' \in L_{\Pi}$, then $\Pi_{\i}'|_{\Sp_2}$ or $\ol{\Pi_{\i}'}|_{\Sp_2}$ is the holomorphic discrete series representation with Blattner parameter $(a+b+3,a+b+3)$ or the non-holomorphic discrete series one with Blattner parameter $(a+3,-b-1)$  which is generic.
For a $\Pi' \in L_{\Pi}$, let 
\[
c_{\G}(\Pi') = \dim_{\C} \{f \in \Pi' \mid \mbox{$f$ is right $\G_{\A}$-invariant} \}.
\]
Let $c_{\G}^{H}(L_{\Pi})$ (resp. $c_{\G}^{W}(L_{\Pi})$) be the sum of $c_{\G}(\Pi')$ over $\Pi' \in L_{\Pi}$ such that $\Pi_{\i}'|_{\Sp_2}$ is the holomorphic (resp. non-holomorphic) discrete series representation with Blattner parameter $(a+b+3,a+b+3)$ (resp. $(a+3,-b-1)$).
First of all, if $\Pi$ is neither a weak endoscopic lift nor a CAP representation, then by Propostion 1.5 of Weissauer \cite{W}, for any $\Pi' \in L_{\Pi}$ there exists $\Pi''$ such that $\Pi_p' \simeq \Pi_p''$ for all 
nonarchimedean $p$ but $\Pi_{\i}' \not \simeq \Pi_{\i}''$ 
and $\Pi_{\i}' \not \simeq \ol{\Pi_{\i}''}$. 
Therefore, $c_{\G}^{H}(L_{\Pi}) = c_{\G}^{W}(L_{\Pi})$ and by TH\'EOR\`EM 7.5 of Laumon \cite{laumon} and Theorem III of \cite{W}, 
\[
L_\Sigma(s,\Pi)^{c_{\G}^{H}(L_{\Pi})} | L_\Sigma(s,{\rm Gr}_w^W H^3_{\rm et}({S_{\G}}_{\bQ},\mathcal{L}^{\rm et}_{a,b}))
\]
where $\Sigma$ is the set of finite places $v$ of 
which $\G_v \not\simeq \Sp_2(\Q_v)$ and we set $L_\Sigma(s,\Pi):= \prod_{v \not\in 
\Sigma} L(s,\Pi_v)$.
However their results of \cite{laumon} and \cite{W} 
do not tell us anything about 
the contribution of a weak endoscopic lift or a CAP representation to the middle cohomology of any Siegel modular threefold. 
On the other hand we have a 
conjectural description of the contribution of 
a weak endoscopic lift $\Pi$ (or also a CAP representation) to the 
(total) $\ell$-adic cohomology of $S_\G$ (Section 6 of \cite{tilouine}).
By Howe, Piatetski-Shapiro \cite{H-PS}, 
there exists an endoscopic lift $\Pi$ 
for a given pair $(\pi_1,\pi_2)$ where $\Pi$ is given by the $\th$-lift from $\GSO_{2,2}$ and globally generic.
Therefore $c_{\G}^W(L_{\Pi}) \neq 0$ for a sufficiently small $\G$.
We should note that 
\begin{itemize}
\item if $\Pi' \in L_{\Pi}$ and $\Pi_{\i} \not\simeq \Pi_{\i}'$, then $\Pi_p \not\simeq \Pi_p'$ for some nonarchimedean $p$.
\item it may be happen $c_{\G}^H(L_{\Pi}) = 0$ for any $\G$ (see  section 5 for the explanation). 
\end{itemize}
By Arthur's conjecture (cf. Section 6 of \cite{tilouine}) and $p$-adic Hodge theory, we guess
\begin{eqnarray*}
L_\Sigma(s,\pi_2)^{c_{\G}^{H}(L_{\Pi})} L_\Sigma(s,\pi_1)^{c_{\G}^{W}(L_{\Pi})} | L_\Sigma(s,{\rm Gr}_w^W H^3_{\rm et}({S_{\G}}_{\bQ},\mathcal{L}^{\rm et}_{a,b})),
\end{eqnarray*}
where $\pi_1$ (resp. $\pi_2$) is chosen so that $\pi_{1,\i}|_{\SL_2}$ (resp. $\pi_{2,\i}|_{\SL_2}$) is the discrete series representation of lowest weight $a-b+2$ (resp. $a+b+4$).
In the following specific case, we would like to give a conjecture.
Let $S$ be a Siegel threefold defined over $\Q$ with a Hecke 
correspondence $\gamma \subset S\times S$ which is also defined 
over $\Q$. 
Assume
\begin{eqnarray}
h^{3,0}(\gamma^\ast{\rm Gr}^W_wH^3_{{\rm betti}}(S_\C,\mathcal{L}_{a,b}))=0,\  
h^{2,1}(\gamma^\ast{\rm Gr}^W_wH^3_{{\rm betti}}(S_\C,\mathcal{L}_{a,b}))=1. \label{eqn:conjhtp}
\end{eqnarray}
Let $\Pi$ be the irreducible cuspidal automorphic representation attached to 
a unique generator of 
$H^{2,1}(\gamma^\ast{\rm Gr}^W_wH^3_{{\rm betti}}(S_\C,\mathcal{L}_{a,b}))$.
By Theorem 2.1 of \cite{Onote}, 
$\Pi_{\i}|_{\Sp_2(\R)}$ is the non-holomorphic discrete series representation with Blattner parameter $(a+3,-b-1)$ which is generic.
By Theorem III and Proposition 1.5 of \cite{W}, $\Pi$ is a CAP representaiton or a weak endoscopic lift of a pair $(\pi_1,\pi_2)$.
But, by section 4 of Schmidt \cite{Sc}, if a Saito-Kurokawa representation (a CAP representation associated to a Siegel parabolically induced representation) is not holomorphic, then its archimedean component is non-tempered and hence 
it is not a discrete series representation.
By Soudry \cite{So}, every CAP representation associated to a Klingen or Borel parabolically induced representation is given by a $\th$-lift from $\GO(L_{\A})$, where $L$ is a quadratic extension of $\Q$.
It is not hard to show that the archimedean component of the $\th$-lift is not generic.
Hence, $\Pi$ is a weak endoscopic lift.
Then our conjecture is:
\begin{conjecture}\label{conj:hodge1type}
Keep the notations as above. 
The following equality of $L$-functions holds up to finitely many local factors:
\[
L(s,\gamma^\ast{\rm Gr}_w^W H^3_{\rm et}({S_{\G}}_{\bQ},\mathcal{L}^{\rm et}_{a,b})) = L(s-b-1,f)
\]
where $f$ is the elliptic cusp form of weight $a-b+2$
 associated to $\pi_1$ (of the lower weight). 
\end{conjecture}
We should explain where the missing contribution of 
$\pi_g \ot \chi^i$ is gone in our case. 
Recall the Galois covering $S_{\Gamma(11)}$ of 
$\mathcal{A}^{{\rm lev}}_{11}$. 
We can construct a holomorphic weak endoscopic lift 
$\Pi' _i= \Th((\pi_f \ot \chi^i) \bt (\pi_g \ot \chi^i))$ 
associated to a pair $(\pi_f\ot\chi^i, \pi_g \ot \chi^i)$ such that 
$\Pi'_i$ has a right $\G(11)_{\A}$-invariant vector which contributes to $(3,0)$-part of 
Gr$^W_3 H^3(S_{\G(11)},\C)$, hence $|L_{\Pi'}| =2$.
The construction of $\Pi'_i$ is given by the first author \cite{okazaki}. 
If we expect Arthur's conjecture (cf. Section 6 of 
\cite{tilouine} and also \cite{hiraga}), $\Pi'_i$ should contribute to the 
$\ell$-adic cohomology of $S_{\Gamma(11)}$ of the middle degree. 
\begin{remark}$($i$)$
The results of this paper can be viewed as a warning that the good behaviour predicted if $\G$ is torsion free may not occur in some case 
like $\K(11)^{{\rm lev}}$, due to the specific geometry of the corresponding Siegel threefold $\mathcal{A}^{{\rm lev}}_{11}$.  
However we know fortunately that 
$\mathcal{A}^{{\rm lev}}_{11}$ is unirational. 
So we can discuss Theorem 1.3-(ii). 
In general, it seems to be hard to do by using (purely) geometric arguments. 

\noindent
$($ii$)$ 
If $\G$ is $inadmissible$ (see Section 5 for the definition and in fact 
$\G=\K(N)$ or $\K(N)^{{\rm lev}}$ for 
any $N$ are such examples), then there are no contribution 
of holomorphic weak endoscopic lifts (hence Yoshida lifts) to 
(3,0)-part of ${\rm Gr}^W_3 H^3(S_{\G},\C)$. 
We will discuss this in Section 5 

\noindent
$($iii$)$ Let $g$ be the elliptic cusp form of weight $a+b+4$
 associated to $\pi_2$. Since the Frobenius eigenvalues at $p\neq \ell$ 
 on $\gamma^\ast{\rm Gr}^W_w H^3_{{\rm et}}(S_{\bQ},\mathcal{L}_{a,b}^{\rm et})$ are of form $p^{b+1}\alpha$ for some $\alpha\in \Z_\ell$, we see that 
 the L-function of $g$ never contribute to 
 $\gamma^\ast{\rm Gr}^W_w H^3_{{\rm et}}(S_{\bQ},\mathcal{L}_{a,b}^{\rm et})$. 

\end{remark}

The paper is organized as follows. 
In section 2, we determine the $L$-function of Klein's cubic hypersurface 
by using theory of Fano threefolds and motives. 
So we will freely use the terminology in \cite{manin} (see also \cite{neko} 
for a modern article).
In section 3, we construct non-holomorphic differential forms $F_i$ on 
$\mathcal{A}^{{\rm lev}}_{11}$. 
In section 4, we give a proof of Theorem 1.3 and discuss a related topic in 
Section 5. 

\section{Klein's cubic threefold and its $L$-function.} 
In this section, we compute $L$-function of Klein's cubic threefold 
$X$ defined by $x_0^2x_1+x^2_1x_2+x^2_2x_3+x^2_3x_4+x^2_4x_0=0$ in 
the projective space 
$\mathbb{P}^4$ where we fix coordinates $[x_0:x_1:x_2:x_3:x_4]$. 
We consider $X$ as a variety defined over $\Q$. 
\begin{prop} The variety $X$ has a good reduction outside 11.
\end{prop}
\begin{proof} We can easily check this by direct computation, but 
we give another proof. 
Let $F_{11}:\displaystyle\sum_{i=0}^4 y^{11}_i=0$ be 
the Fermat hypersurface of degree 11 in $\mathbb{P}^4_\Q$. 
Then we have the generically finite, surjective morphism 
$F_{11}\longrightarrow X, (y_i)_i\mapsto (x_i)_i=(y^4_iy^2_{i+1}y^3_{i+2}y^8_{i+3})_{i\in\Z/5\Z}$ which is defined over $\Z$. 
Clearly $F_{11}$ and the indeterminacy of the map as 
above have good reduction outside 11, hence 
so is $X$. 
\end{proof}

Let $H^\ast_{{\rm dR}}(Y)$ be the algebraic 
de Rham cohomology of a variety $Y$ (may be affine or singular) over a field $K$ 
which is defined to be 
the hypercohomology group of de Rham complex 
$$\Omega^\cdot_Y:=[0\longrightarrow \mathcal{O}_Y
\stackrel{d}{\longrightarrow} \Omega^1_Y
\stackrel{d}{\longrightarrow} \Omega^2_Y 
\longrightarrow\cdots]$$
(see \cite{gro}). It is a $K$-vector space by definition. 
If $Y$ is affine, this coincides with 
the cohomology of the global sections. 
In general, it is hard to compute 
the algebraic de Rham cohomology. However, now 
$X$ is a hypersurface of projective space. 
So we can apply Griffiths-Dwork's results to 
compute an explicit generators of $H^3_{{\rm dR}}(X)$.    

We now explain this. 
Put $S:=x_0^2x_1+x^2_1x_2+x^2_2x_3+x^2_3x_4+x^2_4x_0$ and 
denote by $R=\Q[x_0,x_1,x_2,x_3,x_4]$ the polynomial ring over $\Q$ 
with five variables and $R_d$ the set of all homogeneous 
polynomial of degree $d\in \Z_{\ge 0}$. 
Consider $U:=\mathbb{P}^4\setminus X$. Then $U$ is an affine variety 
over $\Q$ and 
it has coordinate ring $\Gamma(U,\mathcal{O}_U)$ which 
consists of the homogenous elements of degree 0 in 
$R[\frac{1}{S}]$. 

\begin{theorem}(\cite{griffiths}) There exists an isomorphism 
$H^3_{{\rm dR}}(X)\simeq H^4_{{\rm dR}}(U)$ as $\Q$-vector spaces. 
Furthermore, this isomorphism commutes with the action of 
${\rm Aut}(X)\cap {\rm Aut}(\mathbb{P}^4)$.
\end{theorem}
\begin{proof} This follows from the excision theorem and 
the later claim follows from the functoriality of cohomology. 
\end{proof}

So we have only to compute $H^4_{{\rm dR}}(U)$ instead of 
$H^3_{{\rm dR}}(X)$. 
Put $\Omega=\sum_{i=0}^4 x_i dx_0\wedge \cdots \wedge 
\hat{dx_i} \wedge \cdots dx_4$. 

\begin{theorem} The cohomology $H^4_{{\rm dR}}(U)$ 
consists of $\displaystyle\frac{x_i \Omega}{S^2},\frac{\partial_i S 
\Omega}{S^3}$ and $\displaystyle\frac{x_i \Omega}{S^2}$ gives the 
Hodge filtration Fil$^2H^3_{{\rm dR}}(X)$. 
\end{theorem}
\begin{proof} Since $U$ is affine, 
$H^4_{{\rm dR}}(U)={\rm Ker}(d:\Omega^4_U\longrightarrow \Omega^5_U=0)/
{\rm Im}(d:\Omega^3_U\longrightarrow \Omega^4_U)$. 
Furthermore right hand side can be written as 
$$\Big\{\frac{A\Omega}{S^i}\Big|\ A\in R_{3i-5},\ i=2,3,\cdots  \Big\}\Big/
\Big\{\partial_j\Big(\frac{A}{S^i}\Big)\Omega \Big|\ 
j=0,1,2,3,4,\ A\in R_{3i-4},\ i=2,3,
\cdots  \Big\}.$$
By direct computation, we can find generators as in the claim. 
We know a priori that $H^3_{{\rm dR}}(X)$ (resp. 
Fil$^2H^3_{{\rm dR}}(X)$) is 
of dimension 10 (resp. 5)). So this 
would help us from abstract computations.

By \cite{griffiths}, 
Fil$^2H^3_{{\rm dR}}(X)$ corresponds to 
the image of 
$\Big\{\ds\frac{A\Omega}{S^{i+1}}\Big|\ A\in R_{3i-2}, i=1,2,\cdots  \Big\}$ 
in $H^4_{{\rm dR}}(U)$. 
Then we have  the last claim with computations above. 

\end{proof}

Let $\zeta_5$ be a primitive 5-th root of unity.  
\begin{prop} Let $\alpha$ be the automorphism on $X$ 
defined by $[x_0:x_1:x_2:x_3:x_4]\mapsto 
[x_1:x_2:x_3:x_4:x_0]$ ($\alpha$ is of order five) 
and $\alpha^\ast$ be the 
corresponding linear map on $H^3_{{\rm dR}}(X)$. 
Then $H^3_{{\rm dR}}(X)\otimes \Q(\zeta_5)$ 
decomposes as $\bigoplus_{i=0}^4W(\zeta^i_5)$ 
where $\alpha^\ast$ acts on 2-dimensional space 
$W(\zeta^i_5)$ over $\Q(\zeta_5)$ as multiplication $\zeta^i_5$. 
Furthermore, this decomposition preserves Hodge filtration. 
\end{prop}
\begin{proof} By Theorem 2.3, it is easy to see that 
$$v_j=\sum_{i=0}^4 \zeta^{j(i+1)}_5 \frac{x_i\Omega}
{S^2},\ 
\omega_j=\sum_{i=0}^4 \zeta^{j(i+1)}_5 \frac{\partial_i S\Omega}
{S^3},\ j=0,1,2,3,4$$ is a generator of $W(\zeta^i_5)$. 
Since Fil$^2 H^3_{{\rm dR}}(X)\cap W(\zeta^i_5)=\langle v_i \rangle $, 
this decomposition preserves Hodge filtration. 
\end{proof}

\begin{theorem} Let $f$ be the elliptic modular form of weight 2 
with complex multiplication by the ring of 
integers of $\Q(\sqrt{-11})$. Then we have 
$L(s,H^3_{{\rm et}}(X_{\bQ},\Q_\ell))=\prod_{i=0}^4 L(s-1,f\otimes 
\chi^i)$ including  local $L$-factor at 11 where $\chi:(\Z/11\Z)^\ast\longrightarrow \C^\times, 
\overline{2}\mapsto \zeta_5$ 
is a primitive character of order 5 with conductor 11. In particular, 
 $L(s,H^3_{{\rm et}}(X_{\bQ},\Q_\ell))$ is independent to a choice of $\ell$. 
\end{theorem}
\begin{proof} Let $E$ be the elliptic curve over $\Q$ 
which corresponds to $f$. 
Let $S$ be the Hilbert scheme of lines of $X$ which is 
a smooth surface over $\Q$. Then 
by the general theory of Fano threefold (cf.\cite{b&sw}, \cite{manin}), 
the Grothendieck motive $M:=h^3(X)$ associated to $X$ over $\Q$ coincides 
with the motives $h^1(A)(1)$ over $\Q$ associated to the Albanese variety $A$ of $S$ 
where $``(1)"$ means the twist by 
Lefschetz motive. Note that 
$X$ has the Chow-K\"unneth decomposition by 
\cite{manin}. The motive $h^3(X)$ in fact exists in 
the category of Grothendieck motives. 
So we have $$L(s,H^3_{{\rm et}}(X_{\bQ},\Q_\ell))=
L(s,M)=L(s,h^1(A)(1))=L(s-1,H^1_{{\rm et}}(A_{\bQ},\Q_\ell)).$$

It is known by Theorem (46.22) in \cite{adler}  that 
as abelian varieties, 
$A$ is isomorphic to $E^5$ over $\C$. 
Recall $\alpha$ is an automorphism of order 5 defined in Proposition 2.4. 
By functoriality, $\alpha$ is identified with an element 
in ${\rm End}_\C(A)\otimes \Q={\rm End}_\C(E^5)\otimes\Q=
 {\rm M}_5({\rm End}_\C(E) \otimes \Q)= {\rm M}_5(K)$. 
Since $\alpha$ is of order five, this must be a permutation of 
order five such as 
\[
\begin{bmatrix}
0 & 1 & 0 & 0 & 0 \\
0 & 0 & 1 & 0 & 0 \\
0 & 0 & 0 & 1 & 0 \\
0 & 0 & 0 & 0 & 1 \\
1 & 0 & 0 & 0 & 0
\end{bmatrix} \in {\rm M}_5(\Q(\sqrt{-11})).
\]

Therefore, we see that $\alpha-1\in {\rm End}_\Q(A)\otimes\Q$ is a 
non-trivial zero 
divisor and then $B:=A/(\alpha-1)A$ is a one dimensional 
abelian variety over $\Q$. In fact, by Proposition 2.4, 
the filtration of de Rham realization $H^3_{{\rm dR}}(X)_\Q$ of $M$ has 
only one dimensional eigenspace (as a $\Q$-vector space) for eigenvalue 1 of 
$\alpha^\ast$. 
Since $B$ has CM by $\mathcal{O}=\mathcal{O}_K$, $\sharp\mathcal{O}^\times=2$ 
and $B$ has conductor 
a power of 11 by Proposition 2.1, then 
$B$ is isogenous to $E$ or the twist of $E$ by $K$ over $\Q$. 
The latter case becomes $E$ again. 

Consider the quotient abelian variety $B'=A/B$. 
By direct computation, 
we have the local $L$-factor at 3 of 
$X$ and hence of $B\times B'$ up to Tate twists:
\[\begin{array}{rl}
L_3(s,H^3_{{\rm et}}(X_{\bQ},\Q_\ell))=&  (1 + 3x + 27x^2)\times \\
  &(1 - 3x - 18x^2 + 135x^3+
81x^4 +3645x^5 - 13122x^6 - 59049x^7 + 531441x^8),
\end{array}
\]
where $x=3^{-s}$.
Since the second factor of the right hand side is 
irreducible as a polynomial over $\Q$, 
we see that $B'$ is a $\Q$-simple abelian variety.   
We denote by ${\rm End}_\Q(B')$ the ring of endomorphisms defined 
over $\Q$ of $B'$. This is a $\Z$-algebra.
Consider the composite of the following homomorphisms:  
 $$L=\Q(\zeta_5) 
\stackrel{\tiny{\zeta_5\mapsto \alpha}}{\hookrightarrow} {\rm End}_\Q(A)\otimes_\Z\Q=
{\rm End}_\Q(B)\otimes_\Z\Q\times{\rm End}_\Q(B')\otimes_\Z\Q\longrightarrow 
 {\rm End}_\Q(B')\otimes_\Z\Q,$$
where the second homomorphism is the natural projection.  
Since ${\rm End}_\Q(B)\otimes_\Z\Q=\Q$, this map gives an embedding  
$L\hookrightarrow {\rm End}_\Q(B')\otimes_\Z\Q$. 
Then $B'$ is an abelian variety of $\GL_2$-type in the sense of Ribet 
\cite{ribet} and Theorem 4.4 loc.cit 
with Serre's conjecture which is now a theorem by \cite{k&w}, 
$B'$ is isogenous to the Shimura's abelian variety $A_h$ 
for some elliptic modular form $h$ (see Theorem 7.14 of \cite{shimura} for $A_h$) . 

On the other hand, $B'_{\overline{\Q}}$ is isogenous to 
$E^4$ over $\overline{\Q}$. 
Note that $B'$ has good reduction outside 11 by Proposition 2.1. 
Then by Theorem 1.2 in \cite{g&l}, we may assume that 
 $B'_{\overline{\Q}}$ is isogenous to 
$E^4$ over a number field $K$ included in $\Q(\mu_{11^\infty})$. 
Since Gal$(K/\Q)$ is abelian, by taking Weil restriction,  
we must have $h=f\otimes \psi$ for some primitive character $\psi$. 
Note that $L=\Q(a_n(h)|n\ge 1)$ by Theorem 7.14 of \cite{shimura}. 
So we may have $\psi=\chi$. 
Hence we have 
$$L(s,H^3_{{\rm et}}(X_{\bQ},\Q_\ell))=
L(s-1,H^1_{{\rm et}}(A_{\bQ},\Q_\ell))=
\prod_{i=0}^4 L(s-1,H^1_{{\rm et}}(E_{\bQ},\Q_\ell)\otimes \chi^i)
=\prod_{i=0}^4 L(s-1,f\otimes \chi^i).$$
For the local $L$-factor at 11, 
the last equality follows from the local-global compatibility of 
automorphic $L$-function (cf. \cite{carayol},\cite{saito}). 
In particular, the LHS is independent to $\ell$.  
\end{proof}

\section{construction of right $\K(11)_{\A}^{\rm lev}$-invariant cusp forms.}
For $k = \Q,\Q_v$, or $\A$, let 
\[
\GSp_n(k) = \Big\{g \in \GL_{2n}(k) \ \Big |\ {}^tg
\begin{bmatrix}
 0_n& -1_n\\
1_n &0_n
\end{bmatrix}g = c(g) 
\begin{bmatrix}
0_n & -1_n \\
1_n & 0_n
\end{bmatrix},\ c(g)\in k^\times \Big\}
\]
where $c(g)$ is the similitude norm of $g$.
Note that $\GSp_1(k) \simeq \GL_2(k)$.
For a representation $\tau$ of $\GSp_n(k)$ and a quasi-character $\l$, we will denote by $\l\tau$ the representation sending $g$ to $\l(c(g))\tau(g)$.  
For a positive integer $N$, the paramodular groups $\K(N)$ and $\K(N)^{{\rm lev}}$ with a 
canonical level structure are defined by 
\begin{eqnarray*}
\K(N) &=& \left\{\begin{bmatrix}
\Z & \Z& \Z & N\Z \\
N\Z & \Z & N\Z & N\Z \\
\Z & \Z & \Z & N\Z \\
\Z & N^{-1}\Z & \Z & \Z
\end{bmatrix}\right\} \cap \Sp_2(\Q), \\
\K(N)^{{\rm lev}} &=&
\left\{\begin{bmatrix}
\Z & \Z& \Z & N\Z \\
N\Z & 1+N\Z & N\Z & N^2\Z \\
\Z & \Z & \Z & N\Z \\
\Z & \Z & \Z & 1+ N\Z
\end{bmatrix}\right\} \cap \Sp_2(\Q). \label{eqn:paralev}
\end{eqnarray*}
For a nonarchimedean place $v$ of $\Q$, let $\K(N)_v$ and $\K(N)^{{\rm lev}}_v$ be the $v$-adic completions of $\K(N)$ and $\K(N)^{{\rm lev}}$.
Let $\K(N)_\A = \prod_{v < \i} \K(N)_v$ and $\K(N)^{{\rm lev}}_\A = \prod_{v < \i} \K(N)_v^{\rm lev}$. 
In this section, put $p=11$.
Let $\chi = \ot \chi_v$ be a primitive character of $\Q^{\t} \bs \A^{\t}$ of order $5$ with conductor $p$.
Let $f \in S_2(\G_0(11^2))$ be the elliptic CM modular form.
Let $\mu$ be the gr\"o{\ss}encharacter of $K = \Q(\sqrt{-11})$ associated to $f$.
Let 
\[
\nu = \ds\frac{\mu}{|\mu|}
\]
and $\pi_1= \pi(\nu)$ be the irreducible cuspidal automorphic representation of $\PGL_2(\A)$ associated to $\nu$. 
We will construct right $\K(p)_{\A}^{\rm lev}$-invariant non-holomorphic automorphic forms corresponding to non-holomorphic differential forms on $H^{2,1}({\rm Gr}^W_3 H^3_{{\rm betti}}(\mathcal{A}^{{\rm lev}}_{1,11},\C))$.
Let $\Pi$ be an irreducible cuspidal automorphic representation of GSp$_2(\A)$ which arises from a non-holomorphic differential form on $H^{2,1}({\rm Gr}^W_3 H^3_{{\rm betti}}(\mathcal{A}^{{\rm lev}}_{1,11},\C))$.
By the similar argument done before Conjeceture \ref{conj:hodge1type} and Tilouine's conjectural panorama of the occurence of 
automorphic forms in the Hodge decomposition of Siegel threefolds 
(see It\^{o} \cite{ito} or Tilouine, section 6 of \cite{tilouine}), we guess that $\Pi$ is a weak endoscopic lift of a pair $(\chi^i\pi_1,\chi^i\pi_2)$ (see eq.(2) at p. 505 of \cite{o&s} for the case of $a=b=0$).
Here $\pi_2$ is some irreducible cuspidal automorphic representation of $\PGL_2(\A)$ with $\pi_{2,\infty}$ being the discrete series representation of lowest weight $4$. 
We can assume that the central character of $\Pi$ is trivial, since a weak endoscoic lift of $(\chi^i\pi_1,\chi^i\pi_2)$ is a $\chi^i$-twist of that of $(\pi_1,\pi_2)$.
By Roberts \cite{Ro}, every weak endoscopic lift is given by a global $\th$-lift from $\GSO_{\B}(\A)$ for some quaternion algebra $\B$ defined over $\Q$.
We recall some fundamental results on the $\th$-lifts.
Define the action $\rho$ of $\B^{\t} \t \B^{\t}$ on $\B$ by $\rho(h_1,h_2)x = h_1^{-1}xh_2$, which yields an isomorphism 
\[
i_{\rho}: \GSO(\B) \simeq \B^{\t} \t \B^{\t}/\Delta(\Q^{\t})
\]
where $\Delta$ indicates the diagonal embedding.  
Let $\pi_i^{\B}$ be the Jacquet-Langlands transfer of $\pi_i$ to $\B(\A)^{\t}$ if it exists.
By $i_{\rho}$, a pair of automorphic representations $(\pi_1^{\B},\pi_2^{\B})$ of $\B(\A)^{\t}$ is identified with an automorphic representation of $\GSO(\B(\A))$.
We denote by $\Th(\pi_1^{\B} \bt \pi_2^{\B})$ the global $\th$-lift and by $\th(\pi_{1,v}^{\B} \bt \pi_{2,v}^{\B})$ the local $\th$-lift, which is the $v$-component of $\Th(\pi_1^{\B} \bt \pi_2^{\B})$.
If $\B = {\rm M}_2(\Q)$, we write $\Th(\pi_1 \bt \pi_2)$ as $\Th(\pi_1^{\B} \bt \pi_2^{\B})$ and $\th(\pi_{1,v} \bt \pi_{2,v})$ as $\th(\pi_{1,v}^{\B} \bt \pi_{2,v}^{\B})$, briefly.
Then, 
\begin{eqnarray*}
\th(\pi_{1,v}^{\B} \bt \pi_{2,v}^{\B}) \simeq \th(\pi_{1,v} \bt \pi_{2,v}) \ \ \  (\Longleftrightarrow) \ \ \ \B_v \simeq {\rm M}_2(\Q_v).
\end{eqnarray*}
First, we should find a candidate of $\B$ and $\pi_2$ such that $\Pi = \Th(\pi_1^{\B} \bt \pi_2^{\B})$.
If $\B_{\i}$ is not split, then $\th(\pi_{1,\i}^{\B} \bt \pi_{2,\i}^{\B})$ is holomorphic.
Hence, $\B_{\i}$ should be split.
At a nonarchimedean place $v$, $\th(\pi_{1,\i}^{\B} \bt \pi_{2,\i}^{\B})$ is unramified, if and only if $\B_v$ is split and both of $\pi_{1,v}, \pi_{2,v}$ are unramified.
Hence $\pi_{2,v}$ is unramified and $\B_v$ is split for $v \neq p$ since $\Pi_v$ is unramified. 
Therefore, we conclude that $\B = {\rm M}_2(\Q)$ by the Hasse principle, and $\pi_2$ should have a $p$-power conductor.  
Moreover, we know that the local $\th$-lift $\th(\pi_{1,p} \bt \pi_{1,p})$  is a constituent (denoted by $\tau(S,\pi_{1,p})$ in \cite{Ro-Sch}) of the local Klingen parabolically induced representation $1 \rtimes \pi_{1,p}$, and have a right $\G_p$-invariant vector for a relatively large compact subgroup $\G_p \subset \Sp_2(\Q_{p})$.
Further, $\pi(\nu^3)_{p} \simeq \pi(\nu)_{p}$ by the following lemma. Therefore, we choose $\pi_{2} = \pi(\nu^3)$.
\begin{lemma}\label{lem:supB}
The $p$-component $\pi_{1,p} = \pi(\nu)_{p}$ of $\pi_1$ is supercuspidal and $\pi(\nu)_{p} = \pi(\nu^3)_{p}$.
\end{lemma}
\begin{proof}
Let $\p = \sqrt{-11}$, the place of $K$ lying over $p$.
If we assume that $\pi(\nu)_p$ is not supercuspidal, then, by Theorem 4.6. (iii) of \cite{JL}, we can write
\begin{eqnarray}
\nu_{\p} = \chi_{p} \circ N_{K_{\p}/\Q_p} \label{eqn:ass}
\end{eqnarray} 
by some character $\chi_p$ on $\Q_p^{\t}$.
Observing that $[\Z_p^{\t}:N_{K_{\p}/\Q_p}(\o_{\p}^{\t})]=2$ where $\o_{\p}$ is the ring of integers of $K_{\p}$ and 
\begin{eqnarray*}
(\Z_p/p^n\Z_p)^{\t} \cong \Z/p^{n-1}\Z \t \Z/(p-1)\Z,
\end{eqnarray*} 
we conclude that $(\chi_p|_{\Z_p^{\t}})^2=1$.
Take a rational prime $l$ which is inert in $K$ so that $\nu_{w}(l) =1$ at every $w \neq \p, \i$.
Note that there are infinitely many such primes by Dirichlet's arithmetic progression theorem.
Then, $\nu_{\p}(l)=\chi_p(N_{K_{\p}/\Q_p}(l))=\chi_p(l^2)=1$.
Therefore, $L(s,\nu)_l=(1-l^{-2s})^{-1}$, which conflicts to the form of the $l$-factor of $L(s +\tfrac{1}{2},E)$.
Hence $\pi(\nu)_p$ is supercuspidal.

By the theory of complex multiplication, $\mu = \ot_w \mu_w$ takes the values in $K^{\t}$.
Therefore $\nu_{\p}|_{\o_{\p}^{\t}}$ is $\pm 1$-valued since the units group of $K^{\t}$ is $\{ \pm 1 \}$.
Thus $\nu_{\p}|_{\o_{\p}^{\t}} = \nu_{\p}^3|_{\o_{\p}^{\t}}$, and $\nu_{\p}(\v_{\p}) = \pm 1$ or $\pm \sqrt{-1}$ for a uniformizer $\v_{\p}$ of $K_{\p}$.
Hence $\nu_{\p}$ coincides with $\nu_{\p}^3$ or $\ol{\nu}_{\p}^3$.
In any cases,  
\[
\pi_{1,p} = 
\pi(\nu_{\p}) = \pi(\nu_{\p}^3)= \pi(\ol{\nu}_{\p}^3) = \pi_{2,p}.
\]
Here note that central characters of $\pi_1,\pi_2$ are trivial.
This completes the proof.
\end{proof}
\begin{Rem}
By the above argument and Lemma 3.1, the $L$-packet of the weak endoscopic lift associated to $(\pi(\nu),\pi(\nu^3))$ is $\{\Th(\pi(\nu) \bt \pi(\nu^3)), \Th(\pi(\nu)^{\D} \bt \pi(\nu^3)^{\D}) \}$, where $\D$ is the definite quaternion algebra which is not split at only $\i$ and $p$.
\end{Rem}
We should note that if $\Pi$ is a weak endoscopic lift of $(\chi^i \pi_1,\chi^i \pi_2)$, the spinor $L$-function of $\Pi$ has the factor $L(s,\chi^i\pi_2)$ which does not appear in the $L$-function of $X$ (see Theorem 1.3-(ii)). 

Next, we are going to construct a right $\K(p)_{\A}^{\rm lev}$-invariant Whittaker function of $\Th(\pi_1 \bt \pi_2)$. 
Let 
\begin{eqnarray}
H(k) &=& \GL_2(k)^2 \nonumber \\
H^1(k) &=& \{(h_1,h_2) \in H(k) \mid \det(h_1) =\det(h_2) \} \label{def:H^1}
\end{eqnarray}
for $k = \Q, \Q_v$ or $\A$.
We will identify elements of $H(k)$ with those of $\GSO_{{\rm M}_2(k)} = \GSO_{2,2}(k)$ via $i_{\rho}$. 
Let 
\begin{eqnarray*}
e_{1} = \begin{bmatrix}
0 & \frac{1}{p} \\
0 & 0
\end{bmatrix}, \ 
\a = \begin{bmatrix}
\frac{1}{p} &0  \\
0 & -\frac{1}{p}
\end{bmatrix} \in {\rm M}_2(\Q).
\end{eqnarray*}
Let $Z_{(e_1,\a)}(k) \subset H^1(k)$ be the pointwise stabilizer subgroup of $e_{1},\a$, which is isomorphic to 
\begin{eqnarray*}
\bigg\{\big(\begin{bmatrix}
1 & x \\
0 & 1
\end{bmatrix}, 
\begin{bmatrix}
1 & -x \\
0 & 1
\end{bmatrix}\big) \mid x \in k \bigg\}.
\end{eqnarray*} 
We fix the standard additive character $\psi = \ot_v \psi_v$ on $\Q \bs \A$.
Let $f_i$ be an automorphic form in $\pi_i$.
Let $r_v$ be the Weil represetation of $\Sp_2(\Q_v) \t {\rm O}_{2,2}(\Q_v)$ with respect to $\psi_v$ on $\Sc({\rm M}_2(\Q_v)^2)$, the space of Schwartz-Bruhat functions of ${\rm M}_2(\Q_v)^2$.
Then an automorphic form $\th(\vp,f_1 \bt f_2)$ in $\Th(\pi_1 \bt \pi_2)$ is
\begin{eqnarray*}
\th(\vp,f_1 \bt f_2)(g) = \int_{H^1(\Q)\bs H^1(\A)} \sum_{x \in {\rm M}_2(\Q)^2} (\ot_v r_v(g,h)\vp_v(x)) f_1(h_1)f_2(h_2) dh_1dh_2
\end{eqnarray*}
where $\vp_v \in \Sc({\rm M}_2(\Q_v)^2)$ and $d h_1,dh_2$ are Haar measures.
Let $W_i = \ot_v W_{i,v}$ be the global Whittaker function with respect to $\psi$ of $f_i$.
Then, the $v$-component of the standard Whittaker function of $\th(\vp,f_1 \bt f_2)$ is
\begin{eqnarray}
W_v(g) = \int_{Z_{(e_1,\a)}(\Q_v) \bs H^1(\Q_v)} r_v(g,h)\vp_v(e_{1},\a) W_{1,v}(h_1)W_{2,v}(h_2) dh_1dh_2. \label{eqn:W3}
\end{eqnarray}
It is easy to see $W_{\i}(1) \neq 0$ (see Remark \ref{rem:arcsch} for an explicit $\vp_{\i}$).
Let $\vp_v = {\rm Ch}({\rm M}_2(\Z_v)^2)$ for every nonarchimedean place $v \neq p$, where ${\rm Ch}$ indicates the characteristic function.
Then, it is also easy to $W_v(1) \neq 0$, since both of $\pi_{1,v}, \pi_{2,v}$ are unramified and we can take a right $\GL_2(\Z_v)$-invariant $W_{i,v}$.
Consequently, what we have to do is to choose $\vp_{p}$ and $W_{i,p}$ suitably for the realization of a nontrivial right $\K(p)^{\rm lev}$-invariant $W_{p}$.  
Let
\begin{eqnarray*}
\b = W_{1,p}^{\rm new} = W_{2,p}^{\rm new}
\end{eqnarray*}
be the new vector of $\pi_{1,p} \simeq \pi_{2,p}$, which is right $\G_0(p^2)_p$-invariant. 
We can assume $\b(1) = 1$. 
Then it holds 
\begin{eqnarray}
\b(\begin{bmatrix}
a &0  \\
0 & 1
\end{bmatrix}) = 
\begin{cases}
1 & \mbox{if $a \in \Z_{p}^{\t}$} \\
0 & \mbox{otherwise}
\end{cases} \label{eqn:spcusp2}
\end{eqnarray}
since $\pi_{1,p}$ is supercuspidal by Lemma \ref{lem:supB}.
We define  
\begin{eqnarray*}
\vp_{p}^{{\rm lev}}(x_1,x_2)= {\rm Ch} \Big(
\begin{bmatrix}
\Z_p & p^{-1}\Z_p \\
p\Z_p & \Z_p
\end{bmatrix} \oplus 
\begin{bmatrix}
p^{-1}\Z_p^{\t} & p^{-1}\Z_p \\
p\Z_p & p^{-1}\Z_p^{\t}
\end{bmatrix} \Big).
\end{eqnarray*}
By using the properties of the Weil representation in p. 256 of \cite{Ro}, we can check that 
\begin{eqnarray}
r_p(u,(h_1,h_2)) \vp_p^{{\rm lev}} = \vp_p^{{\rm lev}} \label{eqn:vpK-inv}
\end{eqnarray}
for $u \in \K(p)_p^{{\rm lev}}$ and $(h_1,h_2) \in ((\G_0(p^2)_p \t \G_0(p^2)_p) \cap H^1(\Q_p))$.
Now then, let us calculate $W_p(1)$ using (\ref{eqn:W3}).
Let 
\begin{eqnarray*}
\tilde{\G} &=& \begin{bmatrix}
p^2 &0  \\
0 & 1
\end{bmatrix}{\rm GL}_2(\Z_p)\begin{bmatrix}
p^{2} &0  \\
0 & 1
\end{bmatrix}^{-1} \\
&=& 
\begin{bmatrix}
\Z_p & p^{-2}\Z_p \\
p^2\Z_p & \Z_p
\end{bmatrix} \cap {\rm GL}_2(\Q_p) \simeq {\rm GL}_2(\Z_p).
\end{eqnarray*}
As a complete system of representatives for $\tilde{\G}/\G_0(p^2)_p$, we can take
\begin{eqnarray*}
\bigg\{
\begin{bmatrix}
1 & j \\
0 &1 
\end{bmatrix} \mid j \in p^{-2}\Z /\Z \bigg\} \sqcup \bigg\{\begin{bmatrix}
0 & -p^{-1} \\
p & 0
\end{bmatrix}
\begin{bmatrix}
1 & j \\
0 &1 
\end{bmatrix} \mid j \in p^{-1}\Z /\Z \bigg\}.
\end{eqnarray*}
Therefore, as a system of complete representatives of 
$Z_{(e_1,\a)}(\Q_{p}) \bs H^1(\Q_p) /\big(\G_0(p^2)_p \t 
\G_0(p^2)_p \big)$ we can take the following. \\
TYPE I):\begin{eqnarray*}
\bigg(p^r
\begin{bmatrix}
1 & x \\
0& 1
\end{bmatrix}
\begin{bmatrix}
p^m & 0\\
0 & 1
\end{bmatrix}, 
\begin{bmatrix}
p^n & 0 \\
  0 & 1 
\end{bmatrix}
\bigg)
\end{eqnarray*}
with $x \in \Q_p$, $m+2r =n$. \\
TYPE II):\begin{eqnarray*}
\bigg(p^r
\begin{bmatrix}
1 & x \\
0 & 1 
\end{bmatrix}
\begin{bmatrix}
p^m & 0 \\
0   & 1
\end{bmatrix}
\begin{bmatrix}
0   & -1 \\
p^2 & 0
\end{bmatrix}
\begin{bmatrix}
1 & s \\
0 & 1 
\end{bmatrix},
\begin{bmatrix}
p^n & 0 \\
0   & 1
\end{bmatrix}\bigg)
\end{eqnarray*} 
with $x \in \Q_p$, $2r+m+2=n$, and $s \in \{0,\frac{1}{p}, \ldots, \frac{p-1}{p}\}$. \\
TYPE III):
\begin{eqnarray*}
\bigg(p^r
\begin{bmatrix}
1 & x \\
0 & 1 
\end{bmatrix}
\begin{bmatrix}
p^m & 0 \\
0   & 1 
 \end{bmatrix},
\begin{bmatrix}
p^n & 0 \\
0   & 1 
\end{bmatrix}
\begin{bmatrix} 
0   & -1 \\ 
p^2 & 0
\end{bmatrix}
\begin{bmatrix}
1 & t \\ 
0 & 1 
\end{bmatrix}
\bigg)
\end{eqnarray*} 
with $x \in \Q_p$, $m=n+2-2r$ and $t \in \{0,\frac{1}{p}, \ldots, \frac{p-1}{p}\}$. \\
TYPE IV):
\begin{eqnarray*}
\bigg(p^r
\begin{bmatrix}
1 & x \\
0 & 1 
\end{bmatrix}
\begin{bmatrix}
p^m & 0 \\ 
0   & 1 
\end{bmatrix}
\begin{bmatrix}
0   & -1 \\
p^2 & 0
\end{bmatrix}
\begin{bmatrix}
1 & s \\
0 & 1
\end{bmatrix}
,
\begin{bmatrix}
p^n & 0 \\
0   & 1 
\end{bmatrix}
\begin{bmatrix}
0   & -1 \\
p^2 & 0
\end{bmatrix}
\begin{bmatrix}
1 & t \\
0 & 1 
\end{bmatrix}
\bigg)
\end{eqnarray*} 
with $x \in \Q_p$, $2r+ m=n$, and $s,t \in \{0,\frac{1}{p}, \ldots, \frac{p-1}{p}\}$.\\

Let us see the contribution of each type of $h = (h_1,h_2)$ in (\ref{eqn:W3}) to $W_p(1)$. 
We will write 
\[
\rho(h)((e_1,\a))= \bigg(\begin{bmatrix}
a_1 & b_1\\
c_1 & d_1
\end{bmatrix},
\begin{bmatrix}
a_2 & b_2 \\
c_2 & d_2
\end{bmatrix} \bigg).
\]
TYPE I).
If an element $h = (h_1,h_2)$ contributes to $W_p(1)$, at least, $\rho(h)((e_1,\a)) \in {\rm supp}(\vp_p^{{\rm lev}})$ and $\b(h_1)\b(h_2) \neq 0$.
Therefore, by (\ref{eqn:spcusp2}) and (\ref{eqn:vpK-inv}), we can assume  
\begin{eqnarray*}
\bigg(\begin{bmatrix}
0 &  p^{-1-m-r}\\
0 & 0
\end{bmatrix},
\begin{bmatrix}
p^{-1-m+n-r} & p^{-1-m-r}x \\
0 & p^{-1-r}
\end{bmatrix} \bigg) 
\in 
\begin{bmatrix}
\Z_p & p^{-1}\Z_p \\
p\Z_p & \Z_p
\end{bmatrix} \oplus 
\begin{bmatrix}
p^{-1}\Z_p^{\t} & p^{-1}\Z_p \\
p\Z_p & p^{-1}\Z_p^{\t}
\end{bmatrix} 
\end{eqnarray*}
with $m+2r = n,m \ge 0$.
Observing $a_2,d_2$ (resp. $c_1$), we have $r=0,n=m$ (resp. $m \le 0$). 
Thus
\[
\rho(h)((e_1,\a)) \in {\rm supp}(\vp_p^{{\rm lev}}) \Longleftrightarrow m=n=r= 0, x \in \Z_p.
\]
Therefore, the total contribution of this type in (\ref{eqn:W3}) is 
\begin{eqnarray*}
{\rm vol}(\G_0(p^2)_p \t \G_0(p^2)_p) r_p(1,h)\vp_p(e_{1},\a) W_{1,p}(1)W_{2,p}(1) = {\rm vol}(\G_0(p^2)_p \t \G_0(p^2)_p).
\end{eqnarray*}
TYPE II).
By (\ref{eqn:spcusp2}) and (\ref{eqn:vpK-inv}), we can assume
\begin{eqnarray*}
\bigg(\begin{bmatrix}
0 &  p^{-1-m-r}s\\
0 & -p^{-1-m-r}
\end{bmatrix},
\begin{bmatrix}
p^{-1-m+n-r}s & p^{-3-m-r}(-p^m+p^2sx) \\
p^{-1-m+n-r} & -p^{-1-m-r}x
\end{bmatrix} \bigg) \\
\in 
\begin{bmatrix}
\Z_p & p^{-1}\Z_p \\
p\Z_p & \Z_p
\end{bmatrix} \oplus 
\begin{bmatrix}
p^{-1}\Z_p^{\t} & p^{-1}\Z_p \\
p\Z_p & p^{-1}\Z_p^{\t}
\end{bmatrix}
\end{eqnarray*}
with $2r+m+2=n \ge 0$ and $s \in \{0,\frac{1}{p}, \ldots, \frac{p-1}{p}\}$.
Observing $a_2$, we have $s \neq 0$.
Observing $b_1$ and $a_2$, we have $n \le 0$.
Thus $n=0$.
Observing $c_2$, we have $-1-m-r \ge 1$.
Observing $b_2,d_2$, we conclude that, if $\rho(h)(e_1,\a) \in {\rm supp}(\vp_p^{{\rm lev}})$, then 
\begin{eqnarray}
\rho((\begin{bmatrix}
1 & x \\
0 & 1
\end{bmatrix}h_1,h_2))(e_1,\a) \in {\rm supp}(\vp_p^{{\rm lev}}) \label{eqn:keycanc2}
\end{eqnarray}
for any $x \in p^{-1}\Z_p$.
But, by the property 
\begin{eqnarray}
\b(\begin{bmatrix}
1 & x \\
0 & 1
\end{bmatrix} h_1) = \psi_p(x)\b(h_1), \label{eqn:keycanc}
\end{eqnarray}
the contribution of this type is canceled.\\ 
TYPE III).
By (\ref{eqn:spcusp2}) and (\ref{eqn:vpK-inv}), we can assume
\begin{eqnarray*}
\bigg(\begin{bmatrix}
p^{1-m-r} &  p^{1-m-r}t\\
0 & 0
\end{bmatrix},
\begin{bmatrix}
p^{1-m-r}x & p^{-1-m-r}(-p^n+p^2tx) \\
-p^{1-r} & -p^{1-r}t
\end{bmatrix}
\bigg) \\
\in \Big(
\begin{bmatrix}
\Z_p & p^{-1}\Z_p \\
p\Z_p & \Z_p
\end{bmatrix} \oplus 
\begin{bmatrix}
p^{-1}\Z_p^{\t} & p^{-1}\Z_p \\
p\Z_p & p^{-1}\Z_p^{\t}
\end{bmatrix} \Big)
\end{eqnarray*}
with $m=n+2-2r \ge 0$ and $t \in \{0,\frac{1}{p}, \ldots, \frac{p-1}{p}\}$.
Observing $c_2$, we have $r \le 0$.
Thus, $d_2 = -p^{1-r}t$ cannot belong to $p^{-1}\Z_p^{\t}$.
This type does not contribute.\\
TYPE IV).
By (\ref{eqn:spcusp2}) and (\ref{eqn:vpK-inv}), we can assume
\begin{eqnarray*}
\bigg(\begin{bmatrix}
p^{1-m-r}s &  p^{1-m-r}st\\
-p^{1-m-r} & -p^{1-m-r}t
\end{bmatrix},
\begin{bmatrix}
p^{-1-m-r}(-p^m+p^2sx) & -p^{-1-m-r}(p^ns+p^mt-p^2stx) \\
-p^{1-m-r}x & -p^{-1-m-r}(p^n-p^2tx)
\end{bmatrix} \bigg)\\
 \in \Big(
\begin{bmatrix}
\Z_p & p^{-1}\Z_p \\
p\Z_p & \Z_p
\end{bmatrix} \oplus 
\begin{bmatrix}
p^{-1}\Z_p^{\t} & p^{-1}\Z_p \\
p\Z_p & p^{-1}\Z_p^{\t}
\end{bmatrix} \Big).
\end{eqnarray*}
Here $2r+ m=n$ and $s,t \in \{0,\frac{1}{p}, \ldots, \frac{p-1}{p}\}$.
Observing $c_1$, we have $-m-r \ge 0$.
If $-m-r > 0$, then the contribution is canceled by (\ref{eqn:keycanc2}) and (\ref{eqn:keycanc}).
Hence, we can assume $m+r = 0,n = r$.
Then, since $c_2 = -px \in p\Z_p$, we have $x \in \Z_p$.
Since $d_2 = -p^{-1}(p^n-p^2tx) \in p^{-1}\Z_p$, we have
\begin{eqnarray*}
p^{n-1} \in ptx + p^{-1}\Z_p = p^{-1}\Z_p.
\end{eqnarray*}
Thus $n \ge 0$.
Since $a_2= p^{-1}(-p^m+p^2sx) \in p^{-1}\Z_p$, we have
\begin{eqnarray*}
p^{m-1} \in psx + p^{-1}\Z_p = p^{-1}\Z_p.
\end{eqnarray*}
Thus $m \ge 0$, and $n =r = -m \le 0$.
Hence
\[
m=n=r=0.
\]
Under this condition, observing $b_2 \in p^{-1}\Z_p$, we conclude that $s+t \in \Z$. 
Thus the contribution is calculated as
\begin{eqnarray}
c \sum_{y=0}^{p-1}\b(\begin{bmatrix}
0 & -1 \\
p^2 &0 
\end{bmatrix} \begin{bmatrix}
1 & \frac{y}{p}  \\
0 & 1
\end{bmatrix})
\b(\begin{bmatrix}
0 & -1 \\
p^2 &0 
\end{bmatrix}
\begin{bmatrix}
1 & -\frac{y}{p}  \\
0 & 1
\end{bmatrix}) \label{eqn:cont4type}
\end{eqnarray}
with $c= {\rm vol}(\G_0(p^2)_p \t \G_0(p^2)_p)$.
Since $\ep(\frac{1}{2},\pi_{1,p}) = 1$, the eigenvalue of $\b$ for the Atkin-Lehner operetor is $1$ and 
\begin{eqnarray*}
\b(\begin{bmatrix}
0 & -1 \\
p^2 &0 
\end{bmatrix}
\begin{bmatrix}
1 & \frac{y}{p} \\
0 & 1 
\end{bmatrix})
=
\b(\begin{bmatrix}
1 &0  \\
py & 1 
\end{bmatrix}
\begin{bmatrix}
0 & -1 \\
p^2 &0 
\end{bmatrix}
) =
\b(\begin{bmatrix}
1 & 0 \\
py & 1 
\end{bmatrix}
).
\end{eqnarray*}
Let $\W(\pi_{1,p},\psi_p)$ be the Whittaker model of $\pi_{1,p}$ with respect to $\psi_p$.
We define a mapping
\begin{eqnarray*}
C: \W(\pi_{1,p},\psi_p) \ni w(g) \longrightarrow w(\begin{bmatrix}
-1 & 0 \\
0 & 1
\end{bmatrix}
g) = w(\begin{bmatrix}
-1 &0  \\
0 & 1
\end{bmatrix}
g\begin{bmatrix}
-1 &  0\\
 0& 1
\end{bmatrix}
) \in \W(\pi_{1,p},\ol{\psi}_p).
\end{eqnarray*}
By the local newform theory for GL(2), the dimension of the subspace of right $\G_0(p^2)_p$-invariant vectors in $\W(\pi_{1,p},\ol{\psi}_p)$ is one.
Hence, 
\[
C(\b) = \ol{\b}.
\]
Now, (\ref{eqn:cont4type}) is calculated as 
\begin{eqnarray*}
& & c \sum_{y=0}^{p-1}\b(\begin{bmatrix}
1 &0   \\
py & 1
\end{bmatrix}
)
\b(\begin{bmatrix}
1 & 0  \\
-py & 1
\end{bmatrix}
) =
c \sum_{y=0}^{p-1}\b(\begin{bmatrix}
1 &  0 \\
py & 1
\end{bmatrix}
)
\b(\begin{bmatrix}
-1 &  0\\
0 & 1
\end{bmatrix}
\begin{bmatrix}
1 &0   \\
py & 1
\end{bmatrix}
\begin{bmatrix}
-1 &0  \\
 0& 1
\end{bmatrix}
) \\
&=& c \sum_{y=0}^{p-1}\b(\begin{bmatrix}
1 &0   \\
py & 1
\end{bmatrix}
)
C(\b)(\begin{bmatrix}
1 & 0  \\
py & 1
\end{bmatrix}
) = 
c \sum_{y=0}^{p-1}\b(\begin{bmatrix}
1 &  0 \\
py & 1
\end{bmatrix}
)
\ol{\b}(\begin{bmatrix}
1 & 0  \\
py & 1
\end{bmatrix}
) > 0.
\end{eqnarray*}
Hence the total contribution of this type is some positive number.
Combining the contribution of type I), we conclude 
\[
W_{p}(1) \neq 0.
\]
For each $i \in \{0,1,2,3,4\}$ let 
\begin{eqnarray*}
\vp_{p}^{{\rm lev},\chi^i}(x_1,x_2)= \chi^{-i}(\det(p^2x_2)){\rm Ch} \Big(
\begin{bmatrix}
\Z_p & p^{-1}\Z_p \\
p\Z_p & \Z_p
\end{bmatrix}
 \oplus 
\begin{bmatrix}
p^{-1}\Z_p^{\t} & p^{-1}\Z_p \\
\Z_p & p^{-1}\Z_p^{\t}
\end{bmatrix}\Big).
\end{eqnarray*}
Similar to (\ref{eqn:vpK-inv}), $\vp_{p}^{{\rm lev},\chi^i}$ is right  $\K(p)_p^{\rm lev} \t ((\G_0(p^2)_p \t \G_0(p^2)_p) \cap H^1(\Q_p))$-invariant.
For the pair of $(\chi^i\pi_1, \chi^i\pi_2)$ with $1 \le i \le 4$, the proof of the nonvanishing of the local Whittaker function at $p$ is similar to above.
Thus, 
\begin{theorem}
Let $\pi_1 = \pi(\nu), \pi_2 = \pi(\nu^3)$ be the unitary irreducible cuspidal  automorphic representations of $\GL_2(\A)$ associated to the gr\"o{\ss}encharacter $\nu =\frac{\mu}{|\mu|}$. 
Then, each of five globally generic endoscopic lift $\chi^i\Th(\pi_1 \bt \pi_2) = \Th(\chi^i\pi_1 \bt \chi^i\pi_2)$ for $0 \le i \le 4$ has a right $\K(p)_{\A}^{\rm lev}$-invariant automorphic form.
\end{theorem}
\begin{remark}
For a locally generic admissible irreducible representation $\tau$ of $\GSp_2(\Q_v)$, Novodvorsky $\cite{Nov}$ defined a $L$-function of $\tau$.
It holds 
\[
L(s,\chi_v^i \Th(\pi_1 \bt \pi_2)_v) = L(s,\mu \cdot \chi^i \circ N_{K/\Q})_vL(s,\mu^3\cdot 
\chi^i \circ N_{K/\Q})_v.
\]
\end{remark}
\begin{remark}\label{rem:arcsch}
We give a Schwartz-Bruhat function $\vp_{\i} \in \Sc({\rm M}_2(\R)^2)$ for the $\th$-lift of $(\pi_1,\pi_2)$ as follows. 
Set 
\begin{eqnarray*}
P_+(x)= {\rm Tr}(x\begin{bmatrix}
-\sqrt{-1} & -1\\
-1 & \sqrt{-1} 
\end{bmatrix}), \ 
P_-(x)={\rm Tr}(x \begin{bmatrix}
\sqrt{-1} & 1\\
-1 & \sqrt{-1} 
\end{bmatrix})
\end{eqnarray*}
so that $P_{\pm}(\rho(u_{t_1},u_{t_2})x) = e^{-\sqrt{-1} (t_2 \pm t_1)}P_{\pm}(x)$ where $u_{t_i}=\begin{bmatrix}
\cos t_i & \sin t_i \\
-\sin t_i & \cos t_i  
\end{bmatrix} \in {\rm SO}_2(\R)$ for $i=1,2$.
Let $s_1,s_2$ be indeterminants.
Define $\vp_{\i_j} \in \Sc({\rm M}_2(\R)^2) \ot \C[s_1,s_2]$ by
\begin{eqnarray*}
\vp_{\i}(x_1,x_2) = \exp(-\pi \big(\sum_{i=1}^2 a_{i}^2+b_{i}^2+c_{i}^2+d_{i}^2\big))P_+(s_1x_1 + s_2x_2 )^{3} P_-(s_2x_1 - s_1x_2 )
\end{eqnarray*} 
where we write $x_i = \begin{bmatrix}
a_i & b_i \\
c_i & d_i
\end{bmatrix}$.
\end{remark}
\begin{remark}
The local $\th$-lift $\chi^i\th(\pi_{1,p} \bt \pi_{2,p})$ is the irreducible constituent denoted by $\tau(S,\chi^i\pi(\mu)_{p})$ of $1 \rtimes \chi^i\pi_{1,p}$ in $\cite{Ro-Sch}$. 
In particular, the central character $\th(\pi_{1,p} \bt \pi_{2,p})$ is trivial.
According to Roberts, Schmidt $\cite{Ro-Sch}$, $\tau(S,\pi(\mu)_{p})$ has a right $\K(p^4)_p$-invariant Whittaker function, which is the newform.
It is really realized by the $\th$-lift as before with using $\b \in \W(\pi_{1,p},\psi)$ and the  following Schwartz-Bruhat function at $p$: 
\begin{eqnarray*}
\vp_{p}^{{\rm para}}(x_1,x_2) = {\rm Ch} \Big(
\begin{bmatrix}
p\Z_p & p^{-1}\Z_p \\
p^3\Z_p & p\Z_p
\end{bmatrix}
 \oplus p^{-1}{\rm M}_2(\Z_p) \Big).
\end{eqnarray*}
\end{remark}

\section{A comparison of $X$ and $\mathcal{A}^{{\rm lev}}_{1,11}$}
In this section, we  shall discuss the relation of differential forms on $X$ and $\mathcal{A}^{{\rm lev}}_{1,11}$ 
and of $L$-function of these varieties. 

Recall the notations of Section 1.  
Let $\Gamma':=\Gamma(11)$ be the congruence subgroup of level $11$ in 
$\Sp_2(\Z)$ which 
is normal in $g K(11)^{{\rm lev}}g^{-1}$. We denote by $G$ the quotient 
group $gK(11)^{{\rm lev}}g^{-1}/\Gamma'$. Since $G$ is finite, 
the restriction map induces an isomorphism 
of group cohomologies:   
$H^3(\K(11)^{{\rm lev}},\C)\simeq H^3(\Gamma',\C)^G$. 
Since $\K(11)^{{\rm lev}}\backslash\mathbb{H}_2$ (resp. $\Gamma'\backslash\mathbb{H}_2$) 
is an Eilenberg-MacLane space of $\K(11)^{{\rm lev}}$ (resp. $\Gamma'$), 
we have  $H^3(\K(11)^{{\rm lev}},\C)=H^3(\mathcal{A}^{{\rm lev}}_{1,11},\C)$ 
(resp. $H^3(\Gamma',\C)=H^3(S_{\Gamma'},\C)$,\ $S_{\Gamma'}:=\Gamma'\backslash\mathbb{H}_2$) even if 
$\K(11)^{{\rm lev}}$ has torsion elements because we are considering the 
complex  coefficient. 
We note that $S_{\Gamma'}$ is a quasi-projective smooth variety since $\Gamma'$ is torsion free. 
Let $\widetilde{S_{\Gamma'}}$ be a toroidal compactification of $S_{\Gamma'}$ and let 
$j:S_{\Gamma'}\hookrightarrow 
\widetilde{S_{\Gamma'}}$ be the natural inclusion. 
We consider the parabolic cohomology $H^3_!(S_{\Gamma'},\C):={\rm Im}(
H^3(\widetilde{S_{\Gamma'}},\C)\stackrel{j^\ast}{\longrightarrow}H^3(S_{\Gamma'},\C))$. 
Then by observation in Section 7 in \cite{o&s}, we have 
$H^3_{{\rm cusp}}(S_{\Gamma'},\C)=H^3_!(S_{\Gamma'},\C)$. 
Here the cuspidal part $H^3_{{\rm cusp}}(S_{\Gamma'},\C)$ 
(resp. $H^3_{{\rm cusp}}(\mathcal{A}^{{\rm lev}}_{1,11},\C)$) of 
$H^3(S_{\Gamma'},\C)$ (resp. 
$H^3(\mathcal{A}^{{\rm lev}}_{1,11},\C))$ is given in terms of the $(\mathfrak{g},K)$-cohomology 
(cf. Section 2 in \cite{o&s}). That is a complement of the Eisenstein part in 
$H^3(S_{\Gamma'},\C)$ (resp. 
$H^3(\mathcal{A}^{{\rm lev}}_{1,11},\C)$). 
Combining these, we have
$$H^3_{{\rm cusp}}(\mathcal{A}^{{\rm lev}}_{1,11},\C)=H^3_!(S_{\Gamma'},\C)^G
$$ 
and 
$$H^3_{{\rm cusp}}(\mathcal{A}^{{\rm lev}}_{1,11},\C)=
m(\omega_2,K(11)^{{\rm lev}})H^{2,1}(\mathfrak{g},K;H_2)
\oplus m(\omega_3,K(11)^{{\rm lev}})H^{1,2}(\mathfrak{g},K;H_3)$$
by the decomposition (2) at p.505 of \cite{o&s} (see loc.cit. 
for the notation appears here). 
It is easy to see that ${\rm Gr}^W_3 H^3(\mathcal{A}^{{\rm lev}}_{1,11},\C)$ 
contains $H^3_!(S_{\Gamma'},\C)^G=
H^3_{{\rm cusp}}(\mathcal{A}^{{\rm lev}}_{1,11},\C)$.
We hope that the equality  
${\rm Gr}^W_3 H^3(\mathcal{A}^{{\rm lev}}_{1,11},\C)=
H^3_{{\rm cusp}}(\mathcal{A}^{{\rm lev}}_{1,11},\C)$, 
but we do not know if it holds.

We now give a proof of Theorem 1.3.

\begin{proof} The assertion (i) directly follows from the results in Section 2 and Section 3. 
We give a proof of (ii). Hereafter $H^\ast$ means \'etale cohomology and 
we use freely the facts of \'etale cohomology (we refer \cite{milne} for this). 
Since $X$ and $\mathcal{A}^{{\rm lev}}_{1,11}$ is birational to each other, 
we have a common non-empty open subvariety $U$ defined over $\Q$. 
Now we  have exact sequences of compact support \'etale cohomology:
$$\cdots\lra H^3_c(U_{\bQ},\Q_\ell)\lra H^3_c(X_{\Q},\Q_\ell)=H^3(X_{\bQ},\Q_\ell)\lra 
H^3_c((X\setminus U)^{{\rm red}}_{\bQ},\Q_\ell)\lra\cdots$$
$$\cdots\lra H^3_c(U_{\bQ},\Q_\ell)\lra H^3_c({\mathcal{A}^{{\rm lev}}_{1,11}}_{\bQ},\Q_\ell)\lra 
H^3_c((\mathcal{A}^{{\rm lev}}_{1,11}\setminus U)^{{\rm red}}_{\bQ},\Q_\ell)\lra\cdots.$$
Here $(X\setminus U)^{{\rm red}}$ is the closed subscheme $X\setminus U$ with 
the reduced scheme structure (it is same for $(\mathcal{A}^{{\rm lev}}_{1,11}\setminus U)^{{\rm red}}$). 
The difference between $H^3(X_{\bQ},\Q_\ell)$ and $H^3_c({\mathcal{A}^{{\rm lev}}_{1,11}}_{\bQ},\Q_\ell)$ 
are described in terms of $H^3_c((X\setminus U)^{{\rm red}}_{\bQ},\Q_\ell)$ and 
$H^3_c((\mathcal{A}^{{\rm lev}}_{1,11}\setminus U)^{{\rm red}}_{\bQ},\Q_\ell)$. If the closed subschemes
$X\setminus U$ and $\mathcal{A}^{{\rm lev}}_{1,11}\setminus U$ 
containes a scheme $Z$ of dimension less than or equal to one, then the cohomology $H^3_c(Z_{\bQ},\Q_\ell)$ vanishes. So 
we may assume that  these subschemes are unions of surfaces (hence are of dimension 2).  
By Poincare duality, we have $H^3_c((\mathcal{A}^{{\rm lev}}_{1,11} \setminus U)^{{\rm red}}_{\bQ},\Q_\ell)
\simeq H^1((\mathcal{A}^{{\rm lev}}_{1,11} \setminus U)^{{\rm red}}_{\bQ},\Q_\ell)(-1)$. 
Therefore, 
for a sufficiently large $p\not=\ell$, any eigenvalue of
the Frobenius element ${\rm Frob}_p$ acting on $H^3_c((\mathcal{A}^{{\rm lev}}_{1,11}\setminus U)^{{\rm red}}_{\bQ},\Q_\ell)$ 
is of form $p \alpha$ where $\alpha\in \overline{\Z}$ is some Weil number.  
Since $X$ is smooth cubic threefold, the same thing occurs for $H^3(X_{\bQ},\Q_\ell)$ 
as in Proposition 2.5. From this, any eigenvalue of the action of ${\rm Frob}_p$ on 
$H^3_!({\cA^{{\rm lev}}_{1,11}}_{\bQ},\Q_\ell)$ is a multiple of $p$. 
This claims us that $L(s,g)$ does not occur in 
$L(s,H^3_!({\cA^{{\rm lev}}_{1,11}}_{\bQ},\Q_\ell))$.
\end{proof}

\section{Some Remarks.}
Keep the notations in Section 1 which is 
used to state Conjecture \ref{conj:hodge1type}. 
Let $V_\G:={\rm Gr}^W_wH^3_{{\rm betti}}({S_\Gamma}_\C,\mathcal{L}_{a,b})$.
Assume 
\begin{eqnarray}
h^{a+b+3,0}(V_{\G})= h^{0,a+b+3}(V_{\G}) = 0, \ \ h^{a+2,b+1}(V_{\G}) = h^{b+1,a+2}(V_{\G}) \neq 0. \label{eqn:hodgeacc}
\end{eqnarray}
Our moduli space $\cA^{{\rm lev}}_{1,11}$ is an example of 
such a variety for $a= b= 0$. 

Let $\Pi$ be an irreducible cuspidal automorphic representation 
of $\GSp_2(\A)$ which arises from a 
non-holomorphic differential form in $H^{a+2,b+1}(V_{\G})$. 
By the argument before Conjecture \ref{conj:hodge1type}, we guess 
that $\Pi$ should be a weak endoscopic lift associated to a pair 
$(\pi_1,\pi_2)$ so that  
$\pi_{1,\i}|_{\SL_2}$ (resp. $\pi_{2,\i}|_{\SL_2}$) is a discrete series representation of lowest weight $a-b+2$ (resp. $a+b+4$).
We will consider when $V_{\G}$ tends to have the Hodge type (\ref{eqn:hodgeacc}).
If $\Th(\pi_1^{\B} \bt \pi_2^{\B})$ for a quaternion algebra $\B$ contributes to $H^{a+2,b+1}(V_\G)$, then $\Th(\pi_1^{\B} \bt \pi_2^{\B})$ has a right $\G(\A)$-invariant vector and $\B_{\i}$ is split.
On the other hand, if $\B_{\i}$ is not split (i.e., $\B$ is a definite quaternion algebra), then $\Th(\pi_1^{\B} \bt \pi_2^{\B})$ is the so-called Yoshida lift and holomorphic.
In that case, by the Hasse principle, the definite quaternion algebra $\B$ is ramified at some nonarchimedean place $v$.
Here, we should remark that there is no Yoshida lift associated to $(\pi_1,\pi_2)$, if $\pi_1^{\B}$ and $\pi_2^{\B}$ do not exist simultaneously for a common $\B$ (i.e., one of $\pi_{1,v},\pi_{2,v}$ is a principal series representation for every nonarchimedean place $v$).

We say $\G$ is $inadmissible$, if the Weil representation ${r'}^2_v$ of $\Sp_2(\Q_v) \t \O(\D_v)$ for some nonarchimedean place $v$ does not have a right $\G_v$-invariant vector, where $\D_v$ is the unique division quaternion algebra $\D_v$ over $\Q_v$.
We see below that 
$\K(N)$ and $\K(N)^{{\rm lev}}$ are inadmissible for any positive integer $N$. 
For a place $v$, let ${r'}^1_{v}$ (resp. 
${r'}^2_v$) be the Weil representation with respect to some nontrivial additive character $\psi_v$ of ${\rm SL}_2(\Q_{v}) \t {\rm O}(\D_v)$ (resp. $\Sp_2(\Q_{v}) \t {\rm O}(\D_v)$) on $\Sc(\D_v)$ (resp. $\Sc(\D^2_v)$).
It is easy to see that ${r'}^1_v|_{\SL(2)}$ does not have a nontrivial right ${\rm SL}_2(\Z_{v})$-invariant vector for a nonarchimedean place $v$. 
Consider the embedding from ${\rm SL}_2(\Z_{v})$ into $\K(N)_v^{\rm lev}$ via 
\begin{eqnarray}
{\rm SL}_2 \ni \begin{bmatrix}
a & b\\
c & d
\end{bmatrix}
\longmapsto 
\begin{bmatrix}
a & 0 & b & 0 \\
0 & 1 & 0 & 0 \\
c & 0 & d &  0\\
0 & 0 & 0 & 1
\end{bmatrix} \in \Sp_2. \label{eqn:embslsp}
\end{eqnarray}
If ${r'}^2_v|_{\Sp_2}$ has a right 
$\K(N)_v^{\rm lev}$-invariant (or $\K(N)_v^{\rm lev}$-invariant) vector, 
then  ${r'}^1_v|_{\SL_2}$ has a nontrivial right ${\rm SL}_2(\Z_{v})$-invariant vector which gives a contradiction. 
Therefore ${r'}^2_v|_{\Sp_2}$ does not have a right 
$\K(N)_v^{\rm lev}$-invariant (or $\K(N)_v^{\rm lev}$-invariant) vector. 
Therefore we can conclude that there are no contributions of weak 
endoscopic lifts to 
$H^{a+b+3,0}(V_{\G})$ for any inadmissible $\G$, since there is no holomorphic weak endoscopic lift associated to $(\pi_1,\pi_2)$ which has a right $\G(\A)$-invariant vector. Furthermore 
in some case (cf. $\G=K(11)^{{\rm lev}}$ and $a=b=0$), $V_\G$
tends to have the Hodge type (\ref{eqn:hodgeacc}). 

In \cite{o&y}, we gave a conjecture for holomorphic parts of Siegel threefolds.
It can be generalized as follows. This is also along 
the vein of Arthur's conjecture (\cite{arthur},\cite{tilouine}).
\begin{conjecture}
Let $\G \subset \GSp_2(\Q)$ be an arithmetic subgroup.
Suppose $V_\G={\rm Gr}^W_wH^3_{{\rm betti}}({S_{\G}}_\C,\mathcal{L}_{a,b})$ 
has Hodge numbers:
\begin{eqnarray}
h^{a+b+3,0}(V_\G) = h^{0,a+b+3}(V_\G) \neq 0, \ \ h^{a+2,b+1}(V_\G) = h^{b+1,a+2}(V_\G) = 0 \label{eqn:hodgeacc2}
\end{eqnarray}
with $a \ge b \ge 0$.
Suppose $\Pi$ associated to a component of $H^{a+b+3,0}(V_\G)$ and $\Pi$ has multiplicity one (it is so when $h^{a+b+3,0} = 1$).
Then, $\Pi$ is a holomorphic Saito-Kurokawa representation.
\end{conjecture}
This conjecture is true if $\G$ is inadmissible.
Indeed, according to Proposition 1.5 of Weissauer \cite{W}, $\Pi$ is concluded to 
be a CAP representation or a weak endoscopic lift.
Since $\G$ is $inadmissible$, $\Pi$ can not be a weak endoscopic lift by the above argument.
According to Theorem 4.1 of Soudry \cite{So}, every CAP representation associated to a Klingen or Borel parabolically induced representation is given by a $\th$-lift of an irreducible automorphic representation $\tau$ of ${\rm GO}(L_{\A})$ for a quadratic field $L$. 
In case that $L$ is a real quadratic field, every automorphic form $f$ of the $\th$-lift has a nonzero Fourier coefficient associated to $T = {}^tT$ with $\det T \in -d_L (\Q^{\t})^2$, where $d_L$ is the discriminant of $L$ and positive.
Hence $f$ is neither holomorphic nor anti-holomorphic.
Therefore this CAP representation cannot contribute to the $h^{a+b+3,0}, h^{0,a+b+3}$-parts.
In case that $L$ is an imaginary quadratic field, by Theorem 6.13, 7.2 of Kashiwara, Vergne \cite{KV}, the Blattner parameter of the CAP representation associated to $\tau$ is $(c+1,1)$ or $(c+2,2)$ if the weight of $\tau|_{\GSO(L)}$ (identified with a gr\"o{\ss}encharacter of $L$) is $c$.
Hence this CAP representation does not contribute to the $h^{a+b+3,0}$-part.
Thus, $\Pi$ is a holomorphic Saito-Kurokwa representation.


\begin{thebibliography}{20}
\bibitem{adler}A. Adler and S. Ramanan, Moduli of abelian varieties. Lecture Notes in Mathematics, 1644. Springer-Verlag, Berlin, 1996.
\bibitem{arthur}J. Arthur, 
Automorphic representations of ${\rm GSp(4)}$. 
Contributions to automorphic forms, geometry, and number theory, 65-81, 
Johns Hopkins Univ. Press, Baltimore, MD, 2004. 
\bibitem{b&sw}E. Bombieri and H.P.F. Swinnerton-Dyer, 
On the local zeta function of a cubic threefold. 
Ann. Scuola Norm. Sup. Pisa (3) 21 1967 1--29. 
\bibitem{carayol}H. Carayol, Sur les representations $l$-adiques 
associees aux formes modulaires de Hilbert. Ann. Sci. Ecole Norm. Sup. (4) 19 (1986), no. 3, 409-468.
\bibitem{deligne}P. Deligne, La conjecture de Weil. II. Inst. Hautes \'Etudes Sci. Publ. Math. No. 52 (1980), 137--252. 
\bibitem{GT}W. T. Gan and S. Takeda, 
The Local Langlands Conjecture for GSp(4). Ann. of Math. (2) 173 (2011), no. 3, 1841-1882.
\bibitem{g&l}J. Gonz\'alez and J-C. Lario, Modular elliptic directions with complex multiplication 
(with an application to Gross's elliptic curves). Commentarii Math. Hel. 86 (2011), 317--351. 
\bibitem{griffiths}P. Griffiths, On the periods of certain rational integrals. I, Ann. of Math. (2) 90 (1969), 460-495. 
\bibitem{g&p}M. Gross and S. Popescu, The moduli space of $(1,11)$-polarized abelian surfaces is unirational. Compositio Math. 126 (2001), no. 1, 1--23.
\bibitem{gro}A. Grothendieck, 
On the de Rham cohomology of algebraic varieties. Inst. Hautes Etudes Sci. Publ. Math. No. 29 1966 95--103.
\bibitem{H-K}M. Harris and S. Kudla, Arithmetic automophic forms for the nonholomorphic discrete series of GSp$(2)$, Duke Math {\bf 66} (1992), 59-121.
\bibitem{hiraga}K. Hiraga, Endoscopy on GSp$(4)$, The 9th Autumn Workshop on Number Theory, available 
at http://math01.sci.osaka-cu.ac.jp/~furusawa/Hakuba2006/Proceedings.html.  
\bibitem{H-PS}R. Howe and I.I. Piatetski-Shapiro, Some examples of automorphic forms on Sp$_4$, Duke. math {\bf 50} (1983), 55-105.
\bibitem{hkw}K. Hulek, C. Kahn, and S. H. Weintraub, Moduli spaces of abelian surfaces: compactification, degenerations, and theta functions. de Gruyter Expositions in Mathematics, 12. Walter de Gruyter and Co., Berlin, 1993.
\bibitem{ito}T. It\^{o}, 
On motives and l-adic Galois representations associated to automorphic representations of GSp(4), The 9th Autumn Workshop on Number Theory, available 
at http://math01.sci.osaka-cu.ac.jp/~furusawa/Hakuba2006/Proceedings.html.  
\bibitem{JL}H. Jacquet and R. P. Langlands, 
Automorphic forms on GL(2), L.N.M. 114 (1970), Springer.
\bibitem{KV}M. Kashiwara and L. Vergne, On the Segel-Shale-Weil representations and harmonic polynomials, Invent. Math. {\bf 44} (1987), 1-44.
\bibitem{k&w}C. Khare and J-P. Wintenberger, Serre's conjecture I, Invent. Math. 178 (2009), no. 3, 485-504.
\bibitem{K-R-S}S. Kudla, S. Rallis, and D. Soudry, 
On the degree 5 $L$-function for 
Sp(2), Invent. Math. 107 (1992), 483-541.
\bibitem{laumon}G. Laumon, Sur la cohomologie a supports compacts des varietes de Shimura pour 
${\rm GSp}(4)_\Q$. Compositio Math. 105 (1997), no. 3, 267--359.
\bibitem{manin}Ju. I. Manin, Correspondences, motifs and monoidal transformations. Math. USSR-Sb. 6 (1968), 439--470.
\bibitem{milne}J-S. Milne, \'Eatel cohomology,  
Princeton Mathematical Series, 33. Princeton University Press, Princeton, N.J., 1980.
\bibitem{Moriyama}T. Moriyama: Entireness of the spinor $L$-functions for certain generic cusp forms on $GSp(2)$, Amer. J. Math. {\bf 27} (2002), 899-920.
\bibitem{mumford}D. Mumford, J. Fogarty, and F. Kirwan, Geometric invariant theory. Third edition. 
Springer-Verlag, Berlin, 1994. 
\bibitem{neko}J. Nekovar, Beilinson's conjectures. Motives,  537--570, Proc. Sympos. Pure Math., 55, Part 1, Amer. Math. Soc., Providence, RI, 1994. 
\bibitem{Nov}M. Novodvorsky, Automorphic $L$-functions for Symplectic Group GSp(4), Proc. Sympos. Pure Math., {\bf 33}, Amer. Math. Soc., Providence, RI, 1979, 87-95.
\bibitem{Onote}T. Oda, Real Harmonic Analysis for Geometric Automorphic Forms, preprint (available on his homepage). 
\bibitem{o&s}T. Oda and J. Schwermer, Mixed Hodge structures and automorphic forms for Siegel modular varieties of degree two. Math. Ann. 286 (1990), no. 1-3, 481--509.
\bibitem{okazaki}T. Okazaki, On weak endoscopic lift for GSp(2), preprint. 
\bibitem{o&y}T. Okazaki and T. Yamauchi, A Siegel modular threefold and Saito-Kurokawa type lift to 
$S_3(\G_{1,3}(2))$, Math. Ann. 341 (2008), no. 3, 589--601.
\bibitem{Paul}A. Paul, On the Howe correspondence for symplectic-orthogonal 
dual pairs, J. Funct. Anal. 228 (2005), no. 2, 270-310. 
\bibitem{Przebinda}T. Przebinda, The oscillator duality 
correspondence for the pair 
${\rm O}(2,2), {\rm Sp}(2,\R)$, Memoirs of A.M.S 79 Number 403 (1989).
\bibitem{ribet}K. Ribet, Abelian varieties over 
$\mathbb{Q}$ and modular forms. 
Modular curves and abelian varieties, 241--261, 
Progr. Math., 224, Birkhauser, Basel, 2004. 
\bibitem{Ro}B. Roberts, Global $L$-packets for GSp(2) and theta lifts, Docu. math {\bf 6} (2001) 247-314.
\bibitem{Ro-Sch}B. Roberts and R. Schmidt, Local newforms for 
GSp(4), L.N.M. 1918 (2007), Springer.
\bibitem{rou}X. Roulleau, The Fano surface of the Klein cubic threefold, J. Math. Kyoto Univ. Volume 49, 
Number 1 (2009), 113-129. 
\bibitem{saito}T. Saito, Modular forms and $p$-adic Hodge theory. Invent. Math. 129 (1997), no. 3, 607-620.
\bibitem{Sally-Tadic}P. Sally and M. Tadi\'c, Induced rerpesentations and classification for GSp$(2,F)$ and Sp$(2,F)$, Bull. Soc. Math. France {\bf 121} (1993), 75-133.
\bibitem{Sc}R. Schmidt, The Saito-Kurokawa lift and functoriarity, Amer. J. math (2005), 209-240.
\bibitem{So}D. Soudry, The CAP representation of $GSp(4,\A)$, J. reine angrew. Math. {\bf 383} (1988), 87-108.
\bibitem{shimura}G. Shimura, Introduction to the arithmetic theory of automorphic functions. Reprint of the 1971 original. Publications of the Mathematical Society of Japan, 11. Kano Memorial Lectures, 1. Princeton University Press, Princeton, NJ, 1994
\bibitem{G.S}G. Shimura, Abelian Varieties with Complex multiplication and Modular Functions, Princeton univ. press, (1998).
\bibitem{tilouine}J. Tilouine, Cohomologie des varietes de Siegel et representations galoisiennes associees aux representations cuspidales cohomologiques de ${\rm GSp}_4(\Q)$, Algebre et theorie des nombres. Annees 2007-2009, 99-114, Publ. Math. Univ. Franche-Comte Besancon Algebr. Theor. Nr., Lab. Math. Besancon, Besancon, 2009.
\bibitem{V-Z}D. A. Vogan, Jr and G. Zuckerman, Unitary representations with non-zero cohomology, Composito. math. {\bf 53} (1984), 51-90.   
\bibitem{Y}H. Yoshida, Siegel's modular forms and the arithmetics of quadratic forms, Invent. math. {\bf 60} (1980), 193-248.
\bibitem{W}R. Weissauer, Four dimensional Galois representations, Ast\'erisque. {\bf 302}(2005), 67-150.
\end{thebibliography}
\end{document}